\documentclass[12pt,reqno]{amsart}
\usepackage[all]{xy}
\usepackage{amssymb,amsmath,amsthm,mathrsfs}
\usepackage{cases}
\usepackage{enumerate}
\usepackage{hyperref}
\usepackage{dsfont}
\usepackage{tikz}

\numberwithin{equation}{section}
\allowdisplaybreaks[3]
\makeatother
\newtheorem{thm}{Theorem}[section]

\newtheorem{lem}[thm]{Lemma}

\theoremstyle{definition}

\newtheorem{rem}[thm]{Remark}

\numberwithin{equation}{section}

\theoremstyle{definition}
\newtheorem*{ack}{Acknowledgements}

\newcommand{\abs}[1]{\left\lvert#1\right\rvert}
\newcommand{\re}{\textup{Re}}
\newcommand{\im}{\textup{Im}}

\newcommand\cube{\begin{tikzpicture}[scale=2.3]
    \coordinate (A1) at (0, 0);
    \coordinate (A2) at (0, 0.1);
    \coordinate (A3) at (0.1, 0.1);
    \coordinate (A4) at (0.1, 0);
    \coordinate (B1) at (0.03, 0.03);
    \coordinate (B2) at (0.03, 0.13);
    \coordinate (B3) at (0.13, 0.13);
    \coordinate (B4) at (0.13, 0.03);

    \draw (A1) -- (A2);
    \draw (A2) -- (A3);
    \draw (A3) -- (A4);
    \draw (A4) -- (A1);
    \draw[densely dotted] (A1) -- (B1);
    \draw[densely dotted] (B1) -- (B2);
    \draw (A2) -- (B2);
    \draw (B2) -- (B3);
    \draw (A3) -- (B3);
    \draw (A4) -- (B4);
    \draw (B4) -- (B3);
    \draw[densely dotted] (B1) -- (B4);
\end{tikzpicture}}

\newcommand\tinycube{\begin{tikzpicture}[scale=1.3]
    \coordinate (A1) at (0, 0);
    \coordinate (A2) at (0, 0.1);
    \coordinate (A3) at (0.1, 0.1);
    \coordinate (A4) at (0.1, 0);
    \coordinate (B1) at (0.03, 0.03);
    \coordinate (B2) at (0.03, 0.13);
    \coordinate (B3) at (0.13, 0.13);
    \coordinate (B4) at (0.13, 0.03);

    \draw (A1) -- (A2);
    \draw (A2) -- (A3);
    \draw (A3) -- (A4);
    \draw (A4) -- (A1);
    \draw[densely dotted] (A1) -- (B1);
    \draw[densely dotted] (B1) -- (B2);
    \draw (A2) -- (B2);
    \draw (B2) -- (B3);
    \draw (A3) -- (B3);
    \draw (A4) -- (B4);
    \draw (B4) -- (B3);
    \draw[densely dotted] (B1) -- (B4);
\end{tikzpicture}}

\begin{document}

\title[Euler-Kronecker constant of abelian cubic fields]{Extreme values of Euler-Kronecker constants of abelian cubic fields}

\author[Y. Toma]{Yuichiro Toma}
\address{Global Education Center, Waseda University, 1-6-1 Nishiwaseda, Shinjuku-ku, Tokyo 169-8050, Japan.}

\email{yuichiro.toma@aoni.waseda.jp}

\makeatletter
\@namedef{subjclassname@2020}{\textup{2020} Mathematics Subject Classification}
\makeatother

\begin{abstract}
Euler-Kronecker constants are analogues of the Euler-Mascheroni constant for number fields. Assuming the Generalized Riemann Hypothesis for Hecke $L$-functions, we obtain extreme values of Euler-Kronecker constants of abelian cubic fields. We also prove a similar result for quadratic fields.
\end{abstract}

\maketitle
\section{Introduction and statement of results}\label{intro}
\subsection{Introduction}
Let $K$ be a number field. For $\re(s)>1$, the Dedekind zeta-function of $K$ is defined by
\begin{align*}
    \zeta_K(s)= \sum_{\mathfrak{a}} \frac{1}{N(\mathfrak{a})^s} = \prod_{\mathfrak{p}}\frac{1}{1-N(\mathfrak{{p})^{-s}}}.
\end{align*}
Here, $\mathfrak{a}$ (resp. $\mathfrak{p}$) runs over nonzero ideals (resp. prime ideals) in $\mathcal{O}_K$. It can be continued analytically to $\mathbb{C}$ except for $s=1$. The Dedekind zeta function $\zeta_K(s)$ is an important object in number theory because it relates invariants of $K$ via its residue at $s=1$. One of these invariants is known as the Euler-Kronecker constant of $K$, which is defined by
\begin{align*}
    \gamma_K = \lim_{s \to 1} \left( \frac{\zeta_K^\prime}{\zeta_K}(s) +\frac{1}{s-1} \right).
\end{align*}
The case of $K=\mathbb{Q}$, $\gamma_\mathbb{Q}=\gamma$ is the Euler-Mascheroni constant. Let $c_n$ denote the coefficient of Laurent expansion of $\zeta_K(s)$ at $s=1$. The Euler-Kronecker constant can also be written by $c_0/c_{-1}$. 

The Euler-Kronecker constant $\gamma_K$ was first introduced by Ihara~\cite{Ihara1} (Ihara also introduced Euler-Kronecker constants in the function fields settings). 
Following Ihara's pioneering works, the properties of Euler-Kronecker constants have been investigated by many authors. In particular, Euler-Kronecker constants have found applications to the study of other arithmetic invariants, such as the class numbers. 
Dixit~\cite{Dixit} provided a framework showing that suitable bounds for the Euler-Kronecker constants along a tower of number fields imply the Generalized Brauer-Siegel conjecture. Moreover, in the almost normal case, Dixit established the Generalized Brauer-Siegel conjecture by proving an appropriate bound for $\gamma_K$.

In fact, the bound of $\gamma_K$ required to establish the Generalized Brauer-Siegel conjecture is not particularly strong.
The size of $\gamma_K$ has been studied by Ihara~\cite{Ihara1}, who showed that $\gamma_K$ satisfies the following bounds:
\begin{align}
\label{IharaBound}
    -\log \sqrt{\abs{d_K}}+c \leq \gamma_K \leq  \frac{\log \sqrt{\abs{d_K}}+1}{\log \sqrt{\abs{d_K}}-1}(2 \log \log \sqrt{\abs{d_K}}+1)
\end{align}
for some constant $c$, where $d_K$ denotes the discriminant of $K$, and the upper bound is a conditional bound assuming the Generalized Riemann Hypothesis (GRH) for the Dedekind zeta-function of $K$. Tsfasman~\cite{T} showed that the order of lower bound is $ -\log \sqrt{\abs{d_K}}$ up to constant. 

The case of abelian fields, Euler-Kronecker constants are well studied. For cyclotomic fields, Ihara~\cite{Ihara1} showed that $\gamma_{\mathbb{Q}(\zeta_q)} = O((\log q)^2)$ assuming GRH, and this bound was improved to $O(\log q\log\log q)$ by Badzyan~\cite{Bad}. Ihara conjectured that $\gamma_{\mathbb{Q}(\zeta_q)}$ is positive. More precisely, he conjectured that for sufficiently large $q$ it holds that
\begin{align*}
    \left(a_1-\varepsilon \right) \log q <\gamma_{\mathbb{Q}(\zeta_q)}<\left(a_2+\varepsilon \right)\log q
\end{align*}
for $0<a_1,a_2 \leq 2$. On the other hand, Ford, Luca and Moree~\cite{FLM} found a counterexample to Ihara's conjecture; in particular, they showed that $\gamma_{\mathbb{Q}(\zeta_q)} = -0.18237\dots$ when $q=964477901$, and $\gamma_{\mathbb{Q}(\zeta_q)}$ can be negative infinitely often assuming a weak form of the Hardy-Littlewood conjecture. In~\cite{FLM}, they proposed that $\gamma_{\mathbb{Q}(\zeta_q)}$ for a prime $q$ satisfies
\begin{align*}
    \liminf_{q\to\infty} \frac{\gamma_{\mathbb{Q}(\zeta_q)}}{\log q \log \log \log q } =-1 \quad \text{and} \quad \limsup_{q\to\infty} \frac{\gamma_{\mathbb{Q}(\zeta_q)}}{\log q} = 1.
\end{align*}

For another class of abelian extensions, namely quadratic extensions over $\mathbb{Q}$, Euler-Kronecker constants can be expressed explicitly in terms of the logarithmic derivative of quadratic Dirichlet $L$-functions. Let $K= \mathbb{Q}(\sqrt{D})$ denote a quadratic field with a fundamental discriminant $D$. Mourtada and Kumar Murty~\cite{MKM} proved that 
\begin{align*}
\pm \gamma_{\mathbb{Q}(\sqrt{D})} \geq \begin{cases}
    \log \log \abs{D}+O(1) & (\text{unconditionally}) \\
    \log \log \abs{D} +\log \log \log \abs{D} + O(1) & (\text{under GRH}) \\
\end{cases}
\end{align*}
for infinitely many $D$. They also conjectured that the extreme values of $\gamma_{\mathbb{Q}(\sqrt{D})}$, as $D$ ranges over fundamental discriminants with $\abs{D} \leq x$, satisfy
\begin{align*}
    \max_{\abs{D} \leq x} \pm \gamma_{\mathbb{Q}(\sqrt{D})} = \log \log x +\log \log \log x + O(1).
\end{align*}
Assuming GRH, it is possible to prove that such large values exist, namely
\begin{align}
\label{RM-qudratic}
    \max_{\abs{D} \leq x} \pm \gamma_{\mathbb{Q}(\sqrt{D})} \geq \log \log x + \log \log \log x - C_{\text{quad}}^\pm-\varepsilon,
\end{align}
where $C_{\text{quad}}^\pm$ are constants given by
\begin{align}
\begin{split}
\label{RM-Cqud}
    C_{\text{quad}}^+ &= \log 4 - \sum_{p} \frac{\log p}{(p+1)^2(p-1)} + 1 = 2.2541\dots, \\
    C_{\text{quad}}^- &= \log 4 + \sum_{p} \frac{\log p}{(p+1)^2(p-1)} + 1 +2\gamma = 3.6729\dots.
\end{split}
\end{align}
Very recently, Dong and Li~\cite{DL} proved the existence of the large values of $-\gamma_{\mathbb{Q}(\sqrt{D})}$ under GRH with $C_{\text{quad}}^- = 4.1106\dots$. The constant $C_{\text{quad}}^-$ in \eqref{RM-Cqud} improves slightly on the corresponding constant in~\cite{DL}. It is also possible to extend conditional large values to the case of $\gamma_{\mathbb{Q}(\sqrt{D})}$. 
Therefore, we give a proof of the bound \eqref{RM-qudratic} in Section~\ref{quadraticbound}. On the other hand, by the heuristics of Lamzouri~\cite{Lam15}, one could predict that the optimal values of constants are $C_{\text{quad}}^+ = -(A_0 +2\zeta^\prime(2)/\zeta(2)) = 0.3211\dots$ and $C_{\text{quad}}^- =-(A_0-2\gamma)= 0.3356\dots$, where 
\begin{align*}
    A_0= \int_0^1 \tanh y \frac{dy}{y}+ \int_1^\infty (\tanh y-1) \frac{dy}{y} = 0.8187\dots.
\end{align*}

The method of Dong and Li~\cite{DL} relies on the resonance method, which was developed by Soundararajan~\cite{Sou} and Bondarenko and Seip~\cite{BS1, BS2} to obtain large values of the Riemann zeta-function on the critical line. Since then, the resonance methods are widely applied to other number-theoretic objects. More precisely, the technique of Dong--Li and ours is called the long resonance method, which was developed by Aistleitner, Mahatab, Munsch and Peyrot~\cite{AMMP} to find the large values of $\abs{L(1,\chi)}$ of prime moduli. Yang~\cite{Y} adapted the long resonator method for the case of logarithmic derivatives of $L^\prime/L(\sigma,\chi)$ for $1/2 < \sigma \leq 1$. 


\subsection{Main result}
In this paper, we prove the existence of extreme values of Euler-Kronecker constants of abelian cubic fields. By studying extreme values of the logarithmic derivatives of Artin $L$-functions, Cho and Kim~\cite{CK13} proved that there exist infinitely many abelian cubic fields such that 
\begin{align*}
    \gamma_K \leq -2\log \log \abs{d_K} + O(\log\log\log \abs{d_K})
\end{align*}
holds, and there exist infinitely many abelian cubic fields such that 
\begin{align*}
    \gamma_K \geq \log \log \abs{d_K} + O(\log\log\log \abs{d_K})
\end{align*}
holds. They also showed that, under certain assumptions including GRH, the leading terms in these bounds are optimal. 
In the non-abelian case, Mine~\cite{M} studied, in the appendix of~\cite{M}, the limiting distributions of Euler-Kronecker constants of cubic fields over $\mathbb{Q}$ with Galois closure having Galois group $S_3$.
Akbary and Hamieh~\cite{AH} also obtained the limiting distribution results for Euler-Kronecker constants of cubic extensions of $\mathbb{Q}(\omega)$, where $\omega=e^{2\pi i/3}$. 
 
Let $x$ be sufficiently large. Let $\mathcal{A}_3(x)$ denote the set of abelian cubic fields $K$ with $0<d_K \leq x$. 
From the results of Cohn~\cite{Cohn} and Darbar et al.~\cite{DDLL} we have
\begin{align*}
    \#\mathcal{A}_3(x) = \frac{c_3}{2} x^\frac{1}{2} + O\left( x^{\frac{5}{16}+\varepsilon}\right), 
\end{align*}
where $c_3$ is a positive constant given in \eqref{C_3}. Now, we state the main result of this article.

\begin{thm}
\label{LVEKCcubic}
    Assume GRH for Hecke $L$-functions over $\mathbb{Q}(\omega)$. Let $x$ be sufficiently large. Then we have
    \begin{align*}
        \max_{K \in \mathcal{A}_3(x)} \gamma_K \geq  \log \log x + \log \log \log x -C_{\operatorname{cubic}}^+ - \varepsilon
    \end{align*}
    and
    \begin{align*}
        \max_{K \in \mathcal{A}_3(x)} -\gamma_K \geq  2\log \log x + 2\log \log \log x - C_{\operatorname{cubic}}^- - \varepsilon,
    \end{align*}
    where $C_{\operatorname{cubic}}^\pm$ are given by
    \begin{align}
    \begin{split}
    \label{Cintheorem}
        C_{\operatorname{cubic}}^+ &= 3\log 2 + \sum_{p}\frac{(3p+10) \log p}{(p^3-1)(p+2)} + 1 = 5.7972 \dots, \\
        C_{\operatorname{cubic}}^- & = 4 \log 2 + 4 \sum_p \frac{\log p}{(p^3-1)(p+2)}+2+3\gamma = 9.4207\dots.
    \end{split}
    \end{align}
\end{thm}
This is a cubic analogue of \eqref{RM-qudratic}. Assuming GRH, we refine the estimates for the extreme values of $\gamma_K$ of abelian cubic fields $K$ due to Cho and Kim~\cite{CK13} by adding a secondary term. 
Even if we restrict abelian cubic fields from $\mathcal{A}_3(x)$ to its subfamily of cubic fields in which $3$ ramifies (resp. does not ramify), we obtain the same result. 

\section{Cubic characters and cubic $L$-functions}
In this section, we review some basic properties of abelian cubic fields and cubic Dirichlet characters. Let $K$ be an abelian cubic field over $\mathbb{Q}$, so that $[K:\mathbb{Q}]=3$ and $\operatorname{Gal}(K/\mathbb{Q}) \simeq \mathbb{Z}/3\mathbb{Z}$. 
Let $f$ denote the conductor of $K$, which is the smallest positive integer $m$ such that the cyclotomic field $\mathbb{Q}(e^{2\pi i/m})$ contains $K$. 
The discriminant $d_K$ of $K$ is given by $d_K=f^2$. 
It is known, for example~\cite{Cohen, Mayer}, that $f$ is of the form $f = 9^ap_1 \cdots p_t$ with $a=0$ or $1$, depending on whether $3$ is ramified or not in $K$, and $p_i \equiv 1 \pmod 3$. 
There are $2^{a+t-1}$ cyclic cubic fields of discriminant $f$. 
Moreover, if $K$ is an abelian cubic field of conductor $f$, then the Dedekind zeta-function of $K$ can be factorized into
\begin{align*}
    \zeta_K(s) = \zeta(s) L(s,\chi)L(s,\overline{\chi}),
\end{align*}
where $\chi$ is a primitive cubic Dirichlet character conductor $f$. Taking logarithmic derivatives and comparing the constant terms in the Laurent expansions at $s=1$, we have
\begin{align*}
    \gamma_K  = \gamma + 2 \re \frac{L^\prime(1,\chi)}{L(1,\chi)}.
\end{align*}
Since $\gamma = 0.57721\dots$ is bounded, extreme values of $\gamma_K$ are dominated by those values of $L^\prime/L(1,\chi)$. Therefore, if we obtain extreme values of $L^\prime/L(1,\chi)$, where $\chi$ is a primitive cubic Dirichlet character conductor $f$, then the corresponding extreme values of $\gamma_K$ can be obtained for the abelian cubic field $K$ of discriminant $d_K=f^2$.

Now, we turn to find extreme values of the logarithmic derivatives of cubic Dirichlet $L$-functions at $s=1$. Before stating the result, we summarize several basic properties of cubic characters. The following facts are available, for example, at~\cite{BY10, DFK04, DDLL}.

The ring of integers $\mathbb{Z}[\omega]$ of imaginary quadratic field $\mathbb{Q}(\omega)$ is called the Eisenstein ring. Since every ideal of $\mathbb{Z}[\omega]$ is principal, each nonzero ideal $(n) \subseteq \mathbb{Z}[\omega]$ with $(n,3) = 1$ has a unique generator $n \equiv 1 \pmod 3$. We call such $n \in \mathbb{Z}[\omega]$ a primary element in $\mathbb{Z}[\omega]$. 

Let $\mathcal{F}_3$ be the set of primitive cubic Dirichlet characters conductor $f$. 
Every $\chi \in \mathcal{F}_3$ can be written as
\begin{align*}
    \chi= \chi_9^{ae_0}\chi_{p_1}^{e_1} \cdots \chi_{p_t}^{e_t},
\end{align*}
where $a=0$ or $1$, $p_1,\dots,p_t \equiv 1 \pmod 3$, $e_0, e_1,\dots,e_t \in \{1,2\}$, $\chi_9$ is a cubic character of conductor $9$ and $\chi_{p_\ell}$ is a primitive cubic character of conductor $p_{\ell}$. 

It is known that cubic Dirichlet characters correspond to the cubic residue symbol in the Eisenstein ring $\mathbb{Z}[\omega]$. Indeed, any primitive cubic Dirichlet characters of conductor $q$ coprime to $3$ are of the form $m \mapsto \left(\frac{m}{n}\right)_3$ for some $n \in \mathbb{Z}[\omega]$, $n \equiv 1 \pmod 3$, $n$ squarefree and not divisible by any rational primes, with norm $N(n) =q$ (see Lemma 2.1 in~\cite{BY10}). For a rational prime $p \equiv 1 \pmod 3$, $p$ splits in $\mathbb{Z}[\omega]$ to $p=\pi \overline{\pi}$ with $N(\pi) = N(\overline{\pi})=p$. 
Therefore, for $m \in \mathbb{Z}$, a primitive cubic Dirichlet character $\chi$ of conductor $p \equiv 1 \pmod 3$ is given by $\chi(m) = \left( \frac{m}{\pi}\right)_3$, where $ \left( \frac{m}{\pi}\right)_3$ is defined by $\left( \frac{m}{\pi}\right)_3 \equiv m^{\frac{N(\pi)-1}{3}} \pmod \pi$. We find that the other primitive cubic character $\chi^2$ of conductor $p$ is given by $\chi^2(m) = \left( \frac{m}{\overline{\pi}}\right)_3=\overline{\left( \frac{m}{\pi}\right)_3}$.

Since there are exactly two such characters of conductor $p \equiv 1 \pmod 3$, each is the square of the other. Therefore, there are $2^{a+t}$ primitive cubic Dirichlet characters of conductor $f$. Let 
\begin{align*}
    \mathcal{F}_3(z) = \left\{ \chi \in \mathcal{F}_3  \mid \operatorname{cond}(\chi) \leq z \right\}.
\end{align*}
Then, it follows that $\#\mathcal{F}_3(z) = 2 \#\mathcal{A}_3(z^2)$ since $\chi$ and $\chi^2$ correspond to the same abelian cubic field. Hence, up to complex conjugation, $\mathcal{F}_3(z)$ has a one-to-one correspondence to $\mathcal{A}_3(z^2)$.

Now, we state our result, which is a cubic analogue of Yang's result.

\begin{thm}
\label{LVlogderL}
    Let $\chi$ be a primitive cubic Dirichlet character of conductor $f$. Assuming GRH for Hecke $L$-functions over $\mathbb{Q}(\omega)$, we have
    \begin{align*}
        \max_{\chi \in \mathcal{F}_3(z)} \re\frac{L^\prime}{L}(1,\chi) \geq \frac{1}{2}\left(\log \log z + \log \log \log z - \mathfrak{c}_{\operatorname{cubic}}^+ -\varepsilon \right),
    \end{align*}
    and
    \begin{align*}
        \max_{\chi \in \mathcal{F}_3(z)} - \re\frac{L^\prime}{L}(1,\chi) \geq \log \log z + \log \log \log z - \mathfrak{c}_{\operatorname{cubic}}^- -\varepsilon,
    \end{align*}
    where the constants $\mathfrak{c}_{\operatorname{cubic}}^\pm$ are defined by 
    \begin{align*}
    \mathfrak{c}_{\operatorname{cubic}}^+ &= \log 4 + \gamma + 1 + \sum_{p} \frac{(3p+10)\log p}{(p^3-1)(p+2)} = C_{\operatorname{cubic}}^+ +\gamma-\log 2 , \\
    \mathfrak{c}_{\operatorname{cubic}}^- &= \log 4 + \gamma + 1 + 2 \sum_{p} \frac{\log p}{(p^3-1)(p+2)} = \frac{C_{\operatorname{cubic}}^- -\gamma}{2} - \log 2.
    \end{align*}
\end{thm}

Theorem \ref{LVEKCcubic} follows immediately from Theorem \ref{LVlogderL} by taking $z=\sqrt{x}$. Indeed, we have 
\begin{align*}
        \max_{\chi \in \mathcal{F}_3(\sqrt{x})} \left( \gamma+2 \re\frac{L^\prime}{L}(1,\chi) \right) &= 2 \max_{\chi \in \mathcal{F}_3(\sqrt{x})} \re \frac{L^\prime}{L}(1,\chi) + \gamma \\
        &\geq \log \log x + \log \log \log x - C_{\operatorname{cubic}}^+ -\varepsilon,
\end{align*}
and
\begin{align*}
        \max_{\chi \in \mathcal{F}_3(\sqrt{x})} -\left( \gamma+2 \re\frac{L^\prime}{L}(1,\chi) \right) &= 2 \max_{\chi \in \mathcal{F}_3(\sqrt{x})} -\re \frac{L^\prime}{L}(1,\chi)- \gamma \\
        &\geq 2\log \log x + 2\log \log \log x - C_{\operatorname{cubic}}^- -\varepsilon.
\end{align*}

This corresponds to Theorem 6.2 of~\cite{CK13}, which gives unconditional estimates for the extreme values. 
Under GRH, we refine these estimates. 
The leading coefficient of the negative extreme values is consistent with a suggestion of Ihara, Kumar Murty and Shimura (see Remark 1 of \cite{IKMS09}).

Let $\mathcal{F}_{3,a}(z)$ be the set of primitive cubic Dirichlet characters of conductor $9^ap_1 \cdots p_t \leq z$ with $a=0$ or $1$, and $p_j \equiv 1 \pmod 3$. 
Clearly, we have $\mathcal{F}_3(z)= \mathcal{F}_{3,0}(z) \cup \mathcal{F}_{3,1}(z)$. 
We can prove Theorem \ref{LVlogderL} for both $\mathcal{F}_{3,0}(z)$ and $\mathcal{F}_{3,1}(z)$.

The proof of Theorem \ref{LVlogderL} requires estimates for cubic character sums. We can use recent result due to Darbar, David, Lalin and Lumley~\cite{DDLL}. 
The asymptotic formula was already known befere~\cite{DDLL}, for examlple~\cite{Cohn}, \cite{DFK04}.

\begin{lem}
\label{lem:DDLLNC}
Assume GRH for Hecke $L$-functions over $\mathbb{Q}(\omega)$. 
For $m=m_0m_1^3$ we have 
\begin{align*}
    \sum_{\chi \in \mathcal{F}_3(z)} \chi(m) = c_3z \prod_{\substack{p \equiv 1 \bmod 3 \\ p \mid m }}\left( \frac{p}{p+2}\right) \mathds{1}_{m= \tinycube} + O\left( z^{\frac{1}{2}+\varepsilon} g(m_1) f(m_0) \right),
\end{align*}
where $\cube$ denotes cubic integers, 
\begin{align}
    \label{funcf-g}
    g(n) := \sum_{b \mid n} \frac{\mu(b)^2}{b^{\frac{1}{2}+\varepsilon}}, \qquad f(n) := \exp \left( (\log n)^{1-\varepsilon}\right),
\end{align}
$m_0$ (resp. $m_1$) is a cube-free component (resp. cubic component) of $m$ and $c_3$ is a positive constant given by
\begin{align}
\label{C_3}
c_3=\frac{11 \sqrt{3}}{18\pi }\prod_{p \equiv 1 \bmod 3} \left(1-\frac{2}{p(p+1)}\right) =0.3170564 \dots.
\end{align}
\end{lem}

\begin{rem}
Even if we replace $\mathcal{F}_3(z)$ with $\mathcal{F}_{3,a}(z)$, the same asymptotic formula holds up to a difference in the coefficient $c_3$. Indeed, in the case of $\mathcal{F}_{3,0}(z)$ (resp. $\mathcal{F}_{3,1}(z)$), the coefficient becomes $9c_3/11$ (resp. $2c_3/11$).
\end{rem}

This lemma is essentially proven in Proposition 2.1 and Lemma 6.2 of~\cite{DDLL}, however, a treatment of the cube-free part requires careful handling to apply the resonance method. We therefore provide a proof here. 

\begin{proof}
It suffices to show the case of $\chi \in \mathcal{F}_{3,0}$ since if $\chi \in \mathcal{F}_{3,1}$ then 
\begin{align*}
    \sum_{\chi \in \mathcal{F}_{3,1}(z)} \chi(m) &=(\chi_9(m)+\chi_9(m)^2) \sum_{\chi \in \mathcal{F}_{3,0}(z/9)} \chi(m).
\end{align*}
Put $m=m_0m_1^3$ with cube part $m_1$ and cube-free part $m_0$. Since $\chi(m_1^3)= \chi(m_1)^3 =1$ for $\chi\in \mathcal{F}_3(z)$, we have
\begin{align*}
    \sum_{\chi \in \mathcal{F}_{3,0}(z)} \chi(m_0m_1^3) &= \sum_{\substack{\chi \in \mathcal{F}_3(z) \\ (\operatorname{cond}(\chi),m_1)=1}} \chi(m_0) \\  &= \sum_{\substack{\chi \in \mathcal{F}_{3,0}(z)}} \chi(m_0) \sum_{b \mid (\operatorname{cond}(\chi),m_1)} \mu(b) \\
    &= \sum_{\substack{b \mid m_1 \\ b \equiv 1 \bmod 3}} \mu(b) \sum_{\chi_1} \chi_1(m_0)\sum_{\substack{\psi \in \mathcal{F}_{3,0}(z/b) \\ (\operatorname{cond}(\psi),b)=1 }} \psi(m_0) ,
\end{align*}
where, $\chi_1$ is a primitive cubic character modulo $b$ and $\psi$ is a primitive cubic character whose modulus is coprime to $b$. 
For convenience, we put $u =\operatorname{cond}(\chi_1)$. When $m_0>1$, we get
\begin{align*}
    \sum_{\substack{\psi \in \mathcal{F}_{3,0}(z/b) \\ (\operatorname{cond}(\psi),b)=1 }} \psi(m_0) = \sideset{}{^\prime}\sum_{\substack{n \in \mathbb{Z}[\omega] \\ N(n) \leq z/b \\ (n,b)=1}} \left( \frac{m_0}{n} \right)_3,
\end{align*}
where $\sum^\prime$ indicates that $n$ runs over the integers in $\mathbb{Z}[\omega]$ which are square-free, not divisible by any rational primes, and such that $n \equiv 1 \pmod 3$. 

For $n \equiv 1 \pmod 3$, we use the same detectors as~\cite{BY10} and \cite{DDLL} that are
\begin{align*}
\sum_{\substack{ d \in \mathbb{Z} \\ d \mid n \\ d \equiv 1 \bmod 3}} \mu_{\mathbb{Z}}(d) = \begin{cases}
    1 & \text{ if $n$ is not divisible by any rational primes}, \\
    0 & \text{ otherwise},
\end{cases}
\end{align*}
where $\mu_{\mathbb{Z}}(d)=\mu(\abs{d})$ and 
\begin{align*}
\sum_{\substack{ \ell \in \mathbb{Z}[\omega] \\ \ell^2 \mid n \\ \ell \equiv 1 \bmod 3}} \mu_{\mathbb{Z}[\omega]}(\ell) = \begin{cases}
    1 & \text{ if $n$ is squarefree}, \\
    0 & \text{ otherwise},
\end{cases}
\end{align*}
where, $\mu_{\mathbb{Z}[\omega]}$ denotes the M\"obius function on $\mathbb{Z}[\omega]$. Since we choose a primary element, the divisors $d,\ell$ of $n$ can be as this from. 
Substituting the above detectors, we have
\begin{align*}
    \sideset{}{^\prime}\sum_{\substack{n \in \mathbb{Z}[\omega] \\ N(n) \leq z/b \\ (n,b)=1}} \left( \frac{m_0}{n} \right)_3 
    &=\sum_{\substack{d \in \mathbb{Z} \\ d \equiv 1 \bmod 3 \\ (d,b)=1 \\ N(d) \leq z/b}} \mu_{\mathbb{Z}}(d) \left(\frac{m_0}{d} \right)_3 \sum_{\substack{ \ell \in \mathbb{Z}[\omega] \\ \ell \equiv 1 \bmod 3 \\ (\ell ,bd)=1 \\ N(\ell)\leq \sqrt{z/(bN(d))}}} \mu_{\mathbb{Z}[\omega]}(\ell) \left( \frac{m_0}{\ell^2}\right)_3 \sum_{\substack{c \in \mathbb{Z}[\omega] \\ c \equiv 1 \bmod 3 \\ (c,bd)=1 \\ N(c) \leq z/(bN(d\ell^2))}} \left( \frac{m_0}{c} \right)_3.
\end{align*}
The function $\psi_{m_0}: (c) \mapsto \chi_n(m)=\left(\frac{m_0}{c} \right)_3$ defined on ideals $(c) \subseteq \mathbb{Z}[\omega]$ with $(c, 3)=1$ and $m_0 \equiv 1 \pmod 3$ gives a Hecke character of modulus $9m_0$. Let
\begin{align*}
    L(s,\psi_{m_0}) = \sum_{(n)} \frac{\psi_{m_0}((n))}{N(n)^s} = \sum_{\substack{n \in \mathbb{Z}[\omega] \\ n \equiv 1 \bmod 3}} \frac{\chi_n(m)}{N(n)^s}
\end{align*}
be the corresponding Hecke $L$-function. 

Using the conditional bounds of the Hecke $L$-function and Perron's formula, and directly following the proof of Lemma 6.2 of~\cite{DDLL} and Lemma 1 of~\cite{DM}, we have under GRH
\begin{align*}
    \sum_{\substack{c \in \mathbb{Z}[\omega] \\ c \equiv 1 \bmod 3 \\ (c, b)=1 \\ N(c) \leq z/(bN(d\ell^2))}} \left( \frac{m_0}{c} \right)_3 &\ll Y^{\frac{1}{2}+\varepsilon}2^{\omega(bd)} \exp\left( (\log m_0)^{1-\varepsilon}\right)
\end{align*}
with $Y=z/(bN(d\ell^2))$. Since $\omega(n)$ is additive, we have $\omega(db) \leq \omega(d)+\omega(b)$. By using the estimate $\omega(n) \ll \log n/\log\log n$ (see \cite[Theorem 2.10]{MV}), we obtain
\begin{align*}
\sideset{}{^\prime}\sum_{\substack{n \in \mathbb{Z}[\omega] \\ N(n) \leq z/b \\ (n,b)=1}} \left( \frac{m_0}{n} \right)_3 &\ll \left(\frac{z}{b}\right)^{\frac{1}{2}+\varepsilon} 2^{2\omega(b)}\sum_{\substack{d \in \mathbb{Z} \\ d \equiv 1 \bmod 3 \\ (d,b)=1 \\ \abs{d} \leq \sqrt{z/b}}} \frac{2^{\omega(d)}}{d^{1+2\varepsilon}} \sum_{\substack{ \ell \in \mathbb{Z}[\omega] \\ \ell \equiv 1 \bmod 3 \\ (\ell ,b)=1 \\ N(\ell)\leq \sqrt{z}/(bd)}} \frac{1}{N(\ell)^{1+2\varepsilon}} \exp \left((\log m_0)^{1-\varepsilon} \right) \\
    &\quad \ll \left(\frac{z}{b}\right)^{\frac{1}{2}+\varepsilon} 2^{2\omega(b)}\exp \left((\log m_0)^{1-\varepsilon} \right).
\end{align*}
Since the number of primitive cubic character mod $b$ is $2^{\omega(b)}$, by resetting $\varepsilon$, we get
\begin{align*}
    \sum_{\chi \in \mathcal{F}_{3,0}(z)} \chi(m_0m_1^3) &\ll z^{\frac{1}{2}+\varepsilon} \sum_{b \mid m_1} \frac{\mu(b)^2}{b^{\frac{1}{2}+\varepsilon}} \exp \left((\log m_0)^{1-\varepsilon} \right).
\end{align*}

When $m_0=1$, we have
\begin{align*}
     \sum_{\substack{\chi \in \mathcal{F}_{3,0}(z) \\ (\operatorname{cond}(\chi),m_1)=1}} 1 &= \sum_{\substack{b \mid m_1 \\ b \equiv 1 \bmod 3}} \mu(b) 2^{\omega(b)}\sum_{\substack{\psi \in \mathcal{F}_{3,0}(z/b) \\ (\operatorname{cond}(\psi),b)=1 }} 1.
\end{align*}
By the same argument as Proposition 2.1 of \cite{DDLL} assuming GRH, we have
\begin{align*}
    \sum_{\substack{\psi \in \mathcal{F}_{3,0}(z/b) \\ (\operatorname{cond}(\psi),b)=1 }} 1 &= \frac{9c_3}{11}  \prod_{\substack{p \mid b \\ p \equiv 1 \bmod 3 }} \left(1+\frac{2}{p}\right)^{-1} \frac{z}{b} +O\left( 3^{\omega(b)} \left(\frac{z}{b}\right)^{\frac{1}{2}+\varepsilon}\right).
\end{align*}
Therefore, we have
\begin{align*}
     \sum_{\substack{\chi \in \mathcal{F}_{3,0}(z) \\ (\operatorname{cond}(\chi),m_1)=1}} 1 &= \frac{9c_3}{11}z \sum_{\substack{b \mid m_1 \\ b \equiv 1 \bmod 3}} \frac{\mu(b) 2^{\omega(b)}}{b} \prod_{p \mid b} \left(1+\frac{2}{p}\right)^{-1} + O\left( z^{\frac{1}{2}+\varepsilon} \sum_{ b \mid m_1} \frac{\mu(b)^2}{b^{\frac{1}{2}+\varepsilon}}\right) \\
     &= \frac{9c_3}{11} z \prod_{\substack{p \mid m_1 \\ p \equiv 1 \bmod 3}} \left(1 -\frac{2}{p} \left(1+\frac{2}{p}\right)^{-1} \right) + O\left( z^{\frac{1}{2}+\varepsilon} \sum_{ b \mid m_1} \frac{\mu(b)^2}{b^{\frac{1}{2}+\varepsilon}}\right) \\
     &= \frac{9c_3}{11}z \prod_{\substack{p \mid m_1 \\ p \equiv 1 \bmod 3}} \left(\frac{p}{p+2}\right) + O\left( z^{\frac{1}{2}+\varepsilon} \sum_{ b \mid m_1} \frac{\mu(b)^2}{b^{\frac{1}{2}+\varepsilon}}\right).
\end{align*}
Finally, we get
\begin{align*}
     \sum_{\substack{\chi \in \mathcal{F}_{3}(z) \\ (\operatorname{cond}(\chi),m_1)=1}} 1 &= \sum_{\substack{\chi \in \mathcal{F}_{3,0}(z) \\ (\operatorname{cond}(\chi),m_1)=1}} 1 + \sum_{\substack{\chi \in \mathcal{F}_{3,1}(z) \\ (\operatorname{cond}(\chi),m_1)=1}} 1 \\
     &= \left(\frac{9c_3}{11}z+ \frac{9c_3}{11}\frac{2z}{9}\right) \prod_{\substack{p \mid m_1 \\ p \equiv 1 \bmod 3}} \left(\frac{p}{p+2}\right) + O\left( z^{\frac{1}{2}+\varepsilon} \sum_{ b \mid m_1} \frac{\mu(b)^2}{b^{\frac{1}{2}+\varepsilon}}\right) \\
     &= c_3 z \prod_{\substack{p \mid m_1 \\ p \equiv 1 \bmod 3}} \left(\frac{p}{p+2}\right) + O\left( z^{\frac{1}{2}+\varepsilon} \sum_{ b \mid m_1} \frac{\mu(b)^2}{b^{\frac{1}{2}+\varepsilon}}\right)
\end{align*}
as claimed.
\end{proof}

In order to apply the resonance method, an approximate formula for $-L^\prime/L(1,\chi)$ is required. The following is a cubic analogue of Proposition 2.3 in \cite{Lam15} (see also Lemma \ref{lem:lam} bellow).

\begin{lem}
    Let $0 < \delta <1/2$ be fixed, and $d$ be a conductor large. Let $y \geq (\log f)^{\frac{10}{\delta}}$ be a real number. For primitive cubic Dirichlet character $\chi$ conductor $f$ if $L(s, \chi)$  is non-zero for $\re(s)>1-\delta$ and $\abs{\im(s)} \leq y^{\delta}$, then we have
    \begin{align*}
        -\frac{L^\prime}{L}(1,\chi) = \sum_{n \leq y} \frac{\Lambda(n)\chi(n)}{n} + O\left( y^{-\delta/4}\right).
    \end{align*}
\end{lem}

Let $N(\sigma, T, \chi)$ denote the number of zeros of $L(s, \chi)$ in the rectangle $\sigma<\re(s) \leq 1$ and $\abs{\im(s)} \leq T$. Let $\mathcal{E}_\delta(z)$ be the set of exceptional characters $\chi \in \mathcal{F}_3(z)$ which satisfy $N(1-\delta, (\log z)^{\frac{10}{\delta}},\chi)>0$. Then, we have
\begin{align*}
\#\mathcal{E}_\delta(z) = \sum_{\substack{\chi \in \mathcal{F}_3(z) \\ N(1-\delta,(\log z)^{\frac{10}{\delta}}, \chi)>0}} 1 &= \sum_{q \leq z} \sideset{}{^*}\sum_{\substack{\chi \bmod q \\ \chi^3=\chi_0 \\ N(1-\delta,(\log z)^{\frac{10}{\delta}}, \chi)>0}} 1 \leq \sum_{q \leq z} \sideset{}{^*}\sum_{\chi \bmod q} N(1-\delta,(\log z)^{\frac{10}{\delta}}, \chi).
\end{align*}
Here, $\sum^*$ indicates that the sum is restricted to primitive characters. By using Montgomery's zero density estimates~\cite{Mont} we have
\begin{align*}
    \sum_{\substack{\chi \in \mathcal{F}_3(z) \\ N(1-\delta,(\log z)^{\frac{10}{\delta}}, \chi)>0}} 1 \ll \begin{cases}
    z^{\frac{6\delta}{1+\delta}} (\log z)^{\frac{30}{1+\delta}+9} & \text{ if  } \frac{1}{5} \leq  \delta  \leq \frac{1}{2}, \\
    z^{\frac{4\delta}{1-\delta}} (\log z)^{\frac{20}{1-\delta}+14} & \text{ if  } 0 \leq \delta \leq \frac{1}{5}.
    \end{cases}
\end{align*}
Taking $\delta=\frac{1}{10}$, we have $\# \mathcal{E}_\frac{1}{10}(z) \ll z^\frac{4}{9}(\log z)^{O(1)}$. Therefore, for all except $O(z^{\frac{4}{9}+\varepsilon})$ characters of $\mathcal{F}_3(z)$ 
\begin{align*}
    -\frac{L^\prime}{L}(1,\chi) &= \sum_{n \leq y} \frac{\Lambda(n)\chi(n)}{n}+O\left((\log z)^{-5/2} \right)
\end{align*}
holds. 

\section{Proof of Theorem \ref{LVlogderL}; negative extreme values}\label{section:neg}
The proof of Theorem \ref{LVlogderL} relies on the long resonator method. Throughout this section, we assume GRH for Hecke $L$-functions over $\mathbb{Q}(\omega)$. Let $\xi$ be the parameter $\xi =B \log z \log \log z$, where $B$ is a positive constant chosen in \eqref{B}. We define $r(n)$ as a completely multiplicative function by
\begin{align}
\label{coeff-res}
    r(p)=\left( 1- \frac{p}{\xi} \right)
\end{align}
for $p \leq \xi$ and $r(p)=0$ for $p>\xi$. Then for a primitive cubic character $\chi \in \mathcal{F}_3(z)$, we define the resonator by
\begin{align*}
    R(\chi) := \sum_{n=1}^\infty r(n)\chi(n) = \prod_{p \leq \xi} \left(1 - \left(1-\frac{p}{\xi}\right)\chi(p)\right)^{-1}.
\end{align*}

We remark that, without assuming GRH for Hecke $L$-functions over $\mathbb{Q}(\omega)$, we would not be able to use this resonator. In the unconditional case, we have to restrict the support of $r$ to a finite set of integers because the error terms in $\mathcal{M}_{11}$ and $\mathcal{M}_{21}$ (arise from the non-cubic contribution to the cubic character sums; see Lemma 3.5 in \cite{DDLL}) diverge if the coefficients of the resonator $r$ are supported on the set of all $\xi$-smooth integers. Here, we say that an integer $n$ is $\xi$-smooth if $p \mid n \Rightarrow p\leq \xi$. We denote the set of positive $\xi$-smooth integers by $S(\xi)$. This finite-support restriction is a significant disadvantage compared with allowing $r$ to have infinite support.

We also mention that for quadratic characters, assuming GRH for quadratic Dirichlet $L$-functions, Darbar and Maiti~\cite{DM} obtained sharp bounds for the error terms of quadratic character sums. These sharp conditional bounds allow us to enlarge the support of $r$ (see Section~\ref{quadraticbound}).

Put $\mathcal{E} := \mathcal{E}_\frac{1}{10}(z)$ and $y =(\log z)^{100}$ in the previous section. We define
\begin{align*}
    \mathcal{M}_1 := \sum_{\substack{\chi \in \mathcal{F}_3(z) \\ \chi \notin \mathcal{E}}} \re \left(\sum_{n \leq y} \frac{\Lambda(n)\chi(n)}{n} \right)\abs{R(\chi)}^2 \quad \text{ and } \quad \mathcal{M}_2 := \sum_{\substack{\chi \in \mathcal{F}_3(z) \\ \chi \notin \mathcal{E}}} \abs{R(\chi)}^2.
\end{align*}
Then, the large values of $-L^\prime/L(1,\chi)$ can be found by the following:
\begin{align*}
    \max_{\chi \in \mathcal{F}_3(z)} - \re\frac{L^\prime}{L}(1,\chi) \geq \max_{\substack{\chi \in \mathcal{F}_3(z) \\ \chi \notin \mathcal{E}}} - \re\frac{L^\prime}{L}(1,\chi) &\geq \max_{\substack{\chi \in \mathcal{F}_3(z) \\ \chi \notin \mathcal{E}}} \re \left(\sum_{n \leq y} \frac{\Lambda(n)\chi(n)}{n} \right) +o(1) \\
    &\geq \frac{\mathcal{M}_1}{\mathcal{M}_2} 
    + o(1).
\end{align*}

First, by the definition of the resonator, we have 
\begin{align*}
    \abs{R(\chi)} \leq \prod_{p \leq \xi} \left(1-\abs{r(p)} \right)^{-1} = \prod_{p \leq \xi} \frac{\xi}{p}.
\end{align*}
We apply the prime number theorem and partial summation to obtain
\begin{align*}
\begin{split}
\prod_{p \leq \xi} \frac{\xi}{p} &= \exp \left( \log \xi \sum_{p \leq \xi} 1 - \sum_{p \leq \xi} \log p\right) = \exp \left( \frac{\xi}{\log \xi}\left(1+O\left(\frac{1}{\log \xi} \right) \right) \right).
\end{split}
\end{align*}
Therefore, we have
\begin{align}
\label{UBRS}
    \abs{R(\chi)}^2 &\leq  z^{2B\left(1+O\left(1/\log\log z\right)\right)}.
\end{align}
We turn to estimate $\mathcal{M}_1$. We divide $\mathcal{M}_1$ into
\begin{align*}
    \mathcal{M}_1 &= \sum_{\chi \in \mathcal{F}_3(z)} \sum_{n \leq y} \frac{\Lambda(n)\chi(n)}{n} \abs{R(\chi)}^2 - \sum_{\substack{\chi \in \mathcal{F}_3(z) \\ \chi \in \mathcal{E}}} \sum_{n \leq y} \frac{\Lambda(n)\chi(n)}{n} \abs{R(\chi)}^2.
\end{align*}
We denote the right-hand side by $\mathcal{M}_{11} - \mathcal{M}_{12}$. By \eqref{UBRS}, we have
\begin{align}
\label{M12}
    \mathcal{M}_{12} &\ll \# \mathcal{E} \sum_{n \leq y} \frac{\Lambda(n)}{n} \abs{R(\chi)}^2 \ll z^{\frac{4}{9}+2B+2\varepsilon}.
\end{align}

Since $\overline{\chi(n)} = \chi(n^2)$ for $\chi \in \mathcal{F}_3$, we have 
\begin{align*}
    \mathcal{M}_{11} &= \sum_{n \leq y} \frac{\Lambda(n)}{n} \sum_{k=1}^\infty \sum_{l=1}^\infty r(k)r(l) \sum_{\chi \in \mathcal{F}_3(z)} \chi(nkl^2).
\end{align*} 

By applying Lemma \ref{lem:DDLLNC} under GRH, we get
\begin{align*}
    \mathcal{M}_{11} &= z c_3 \sum_{n \leq y} \frac{\Lambda(n)}{n} \sum_{\substack{k,l=1 \\ nkl^2 = \tinycube}}^\infty r(k)r(l) \prod_{p \mid nkl^2} \left( \frac{p}{p+2}\right) \\
    &\quad + O\left( z^{\frac{1}{2}+\varepsilon} \sum_{n \leq y}\frac{\Lambda(n)}{n} \sum_{\substack{k, l=1 \\ nkl^2 = \tinycube}}^\infty r(k)r(l) \sum_{b \mid nkl^2} \frac{\mu(b)^2}{b^{\frac{1}{2}+\varepsilon}} \right) \\
    &\quad + O\left( z^{\frac{1}{2}+\varepsilon} \sum_{n \leq y}\frac{\Lambda(n)}{n} \sum_{\substack{k, l=1 \\ nkl^2 \neq \tinycube}}^\infty r(k)r(l) \sum_{b \mid s_1} \frac{\mu(b)^2}{b^{\frac{1}{2}+\varepsilon}} \exp\left( (\log s_0)^{1-\varepsilon} \right) \right),
\end{align*} 
where $s_1$ (resp. $s_0$) is cube part (resp. cube-free part) of $nkl^2$.

Since the function $g(n)$ is multiplicative and $g(p^n)=g(p)$, it holds that $g(nkl^2) \leq g(n)g(k)g(l)$. Hence, we have
\begin{align*}
    &z^{\frac{1}{2}+\varepsilon} \sum_{n \leq y}\frac{\Lambda(n)}{n} \sum_{\substack{k, l=1 \\ nkl^2 =\tinycube}}^\infty r(k)r(l) \sum_{b \mid nkl^2} \frac{\mu(b)^2}{b^{\frac{1}{2}+\varepsilon}} \ll z^{\frac{1}{2}+\varepsilon} \sum_{n \leq y}\frac{\Lambda(n)g(n)}{n} \left( \sum_{k=1}^\infty r(k)g(k) \right)^2.
\end{align*}
Then, the prime number theorem implies
\begin{align*}
    \sum_{n \leq y} \frac{\Lambda(n)g(n)}{n} &= \sum_{p \leq y} \frac{\log p}{p} \left( 1+ \frac{1}{p^{\frac{1}{2}+\varepsilon}}\right) +O(1) \ll \log\log z,
\end{align*}
and
\begin{align*}
    \sum_{k=1}^\infty r(k) g(k) &= \prod_{p \leq \xi} \left( \frac{\xi}{p} \left( 1+ \frac{1}{p^{\frac{1}{2}+\varepsilon}}-\frac{p^{\frac{1}{2}-\varepsilon}}{\xi}\right) \right) \\
    &= \exp \left( \pi(\xi) \log \xi - \sum_{p \leq \xi} \log p + \sum_{p \leq \xi}\log \left( 1+\frac{1}{p^{\frac{1}{2}+\varepsilon}}-\frac{p^{\frac{1}{2}-\varepsilon}}{\xi}\right) \right) \\
    &\ll z^{B\left(1+O\left(1/\log\log z\right)\right)}.
\end{align*}
Therefore, 
\begin{align*}
z^{\frac{1}{2}+\varepsilon} \sum_{n \leq y}\frac{\Lambda(n)}{n} \sum_{\substack{k, l=1 \\ nkl^2 =\tinycube}}^\infty r(k)r(l) \sum_{b \mid nkl^2} \frac{\mu(b)^2}{b^{\frac{1}{2}+\varepsilon}} &\ll z^{\frac{1}{2}+2B + 4\varepsilon}.
\end{align*}

For the second error term, let $n_0,k_0,l_0$ denote the squarefree part of $n,k,l$, respectively. By following \cite{DM}, we use the inequality $f(s_0) \leq f(n_0k_0l_0^2) \leq f(n_0)f(k_0) f(l_0^2)$ to obtain
\begin{align*}
& z^{\frac{1}{2}+\varepsilon} \sum_{n \leq y}\frac{\Lambda(n)}{n} \sum_{\substack{k, l=1 \\ nkl^2 \neq \tinycube}}^\infty r(k)r(l) \sum_{b \mid s_1} \frac{\mu(b)^2}{b^{\frac{1}{2}+\varepsilon}} \exp\left( (\log s_0)^{1-\varepsilon} \right) \\
&\ll z^{\frac{1}{2}+\varepsilon} \sum_{n \leq y}\frac{\Lambda(n)g(n) f(n_0)}{n} \left(\sum_{k=1}^\infty r(k)g(k)f(k_0^2) \right)^2.
\end{align*}
From the support condition of $r$, we find that $k_0 \leq \prod_{p \leq \xi} p^2$, and hence we have 
\begin{align*}
    f(k_0^2) = \exp\left( (2 \log k_0)^{1-\varepsilon}\right) \leq \exp\left( \left(4 \sum_{p \leq \xi} \log p \right)^{1-\varepsilon}\right) \leq \exp \left( \left( 4B \log z \log \log z \right)^{1-\varepsilon} \right).
\end{align*}
Similar to the evaluation of the first error term, we have
\begin{align*}
\sum_{n \leq y}\frac{\Lambda(n)g(n) f(n_0)}{n} &= \sum_{p \leq y} \frac{\log p}{p} \left(1+\frac{1}{p^{\frac{1}{2}+\varepsilon}} \right)\exp \left((\log p)^{1-\varepsilon} \right) +O(1)\\
&\ll \exp \left((\log y)^{1-\varepsilon}\right) \sum_{p \leq y} \frac{\log p}{p} \left(1+\frac{1}{p^{\frac{1}{2}+\varepsilon}} \right) \\
&\ll \log z.
\end{align*}
Therefore, 
\begin{align*}
 z^{\frac{1}{2}+\varepsilon} \sum_{n \leq y}\frac{\Lambda(n)}{n} \sum_{\substack{k, l=1 \\ nkl^2 \neq \tinycube}}^\infty r(k)r(l) \sum_{b \mid s_1} \frac{\mu(b)^2}{b^{\frac{1}{2}+\varepsilon}} \exp\left( (\log s_0)^{1-\varepsilon} \right) 
&\ll z^{\frac{1}{2}+2B + 5\varepsilon}.
\end{align*}
Summarizing the above, we obtain
\begin{align}
\begin{split}
    \label{M11}
    \mathcal{M}_{11} &= z c_3 \sum_{n \leq y} \frac{\Lambda(n)}{n} \sum_{\substack{k,l=1 \\ pkl^2 = \tinycube}}^\infty r(k)r(l) \prod_{p\mid nkl^2} \left( \frac{p}{p+2}\right) +O\left(z^{\frac{1}{2}+2B + 5\varepsilon} \right) .
    \end{split}
\end{align}

Next, we analyze $\mathcal{M}_2$. Using the same argument as $\mathcal{M}_1$, we divide
\begin{align*}
    \mathcal{M}_2 &= \sum_{\chi \in \mathcal{F}_3(z)} \abs{R(\chi)}^2 - \sum_{\substack{\chi \in \mathcal{F}_3(z) \\ \chi \in \mathcal{E}}}  \abs{R(\chi)}^2.
\end{align*}
We denote the right-hand side by $\mathcal{M}_{21} - \mathcal{M}_{22}$. Using $\# \mathcal{E} \ll z^{\frac{4}{9}+\varepsilon}$ and \eqref{UBRS}, we have
\begin{align}
\label{M22}
    \mathcal{M}_{22} &\ll z^{\frac{4}{9}+2B+\varepsilon}.
\end{align}

For $\mathcal{M}_{21}$, Lemma \ref{lem:DDLLNC} under GRH implies that
\begin{align*}
    \mathcal{M}_{21} &= z c_3 \sum_{\substack{k,l=1 \\ kl^2 = \tinycube}}^\infty r(k)r(l) \prod_{p \mid kl^2} \left( \frac{p}{p+2}\right) + O\left( z^{\frac{1}{2}+\varepsilon} \sum_{\substack{k, l=1 \\ kl^2 = \tinycube}}^\infty r(k)r(l) \sum_{b \mid kl^2} \frac{\mu(b)^2}{b^{\frac{1}{2}+\varepsilon}} \right) \\
    &\quad + O\left( z^{\frac{1}{2}+\varepsilon} \sum_{\substack{k, l=1 \\ kl^2 \neq \tinycube}}^\infty r(k)r(l) \sum_{b \mid u_1} \frac{\mu(b)^2}{b^{\frac{1}{2}+\varepsilon}} \exp\left( (\log u_0)^{1-\varepsilon} \right) \right),
\end{align*} 
where $u_1$ (resp. $u_0$) is cube part (resp. cube-free part) of $kl^2$.
By the same argument as $\mathcal{M}_{11}$, we have
\begin{align*}
z^{\frac{1}{2}+\varepsilon} \sum_{\substack{k, l=1 \\ kl^2 = \tinycube}}^\infty r(k)r(l) \sum_{b \mid kl^2} \frac{\mu(b)^2}{b^{\frac{1}{2}+\varepsilon}} &\ll z^{\frac{1}{2}+2B + 3\varepsilon},
\end{align*}
and 
\begin{align*}
    z^{\frac{1}{2}+\varepsilon} \sum_{\substack{k, l=1 \\ kl^2 \neq \tinycube}}^\infty r(k)r(l) \sum_{b \mid u_1} \frac{\mu(b)^2}{b^{\frac{1}{2}+\varepsilon}} \exp\left( (\log u_0)^{1-\varepsilon} \right) &\ll z^{\frac{1}{2}+2B + 4\varepsilon}.
\end{align*}

Therefore, we obtain
\begin{align}
    \begin{split}
    \label{M21}
    \mathcal{M}_{21} &= z c_3 \sum_{\substack{k,l=1 \\ kl^2 = \tinycube}}^\infty r(k)r(l) \prod_{p \mid kl^2} \left( \frac{p}{p+2}\right) +O\left( z^{\frac{1}{2}+2B + 4\varepsilon} \right).
    \end{split}    
\end{align}
Now, the constant $B$ can be chosen as
\begin{align}
\label{B}
    B=\frac{e^{-12\varepsilon}}{4}
\end{align}
to satisfy $\frac{4}{9}+2B+2\varepsilon, \frac{1}{2}+2B + 5\varepsilon<1$.

Therefore by \eqref{M12}, \eqref{M11}, \eqref{M22} and \eqref{M21}, we obtain
\begin{align}
\label{ratio}
    \frac{\mathcal{M}_1}{\mathcal{M}_2} = \frac{\mathcal{M}_{11}-\mathcal{M}_{12}}{\mathcal{M}_{21}-\mathcal{M}_{22}}=\frac{\displaystyle \sum_{n \leq y} \frac{\Lambda(n)}{n} \sum_{\substack{k,l=1 \\ nkl^2 = \tinycube}}^\infty r(k)r(l) \prod_{p \mid nkl^2} \left( \frac{p}{p+2}\right)}{\displaystyle \sum_{\substack{k,l=1 \\ kl^2 = \tinycube}}^\infty r(k)r(l) \prod_{p \mid kl^2} \left( \frac{p}{p+2}\right)}\left(1 + o\left( 1\right)\right).
\end{align}
To estimate the above ratio, following~\cite{DDLL}, we decompose $k$ as $k=k_1k_2^2k_3^3k_4^3$, where $k_1, k_2, k_3$ are square-free and coprime in pairs. Similarly, we also decompose $l$ as $l=l_1l_2^2l_3^3l_4^3$, where $l_1,l_2,l_3$ are as well. In this decomposition, for prime $p$ we get $p \mid k_4 \Rightarrow p \mid k_1k_2k_3$ and $p \mid l_4 \Rightarrow p \mid l_1l_2l_3$. We also write $n=q^m$. By these decompositions, we find that if $q^mkl^2= \cube$ (resp. $kl^2 = \cube$), then $q^mk_1k_2^2 l_1^2 l_2=\cube$ (resp. $k_1k_2^2l_1^2l_2 = \cube$). Hence, we have
\begin{align}
\begin{split}
\label{decom}
    & \sum_{\substack{k,l=1 \\ q^mkl^2 = \tinycube}}^\infty r(k)r(l) \prod_{p \mid q^mkl^2} \left( \frac{p}{p+2}\right) \\
    &= \sum_{\substack{k,l=1 \\ q^mkl^2 = \tinycube}}^\infty r(k)r(l) \prod_{p \mid q^mkl^2} \left( \frac{p}{p+2}\right) \\ 
    &=\sum_{k_1,l_1 \in S(\xi)} \mu(k_1)^2\mu(l_1)^2 \sum_{\substack{k_2, l_2 \in S(\xi) \\ (k_2,k_1)=1 \\ (l_2,l_1)=1 \\ q^mk_1k_2^2l_1^2l_2=\cube}} \mu(k_2)^2 \mu(l_2)^2 \sum_{\substack{k_3, l_3 \in S(\xi) \\ (k_3,k_1k_2)=1 \\ (l_3,l_1l_2)=1}} \mu(k_3)^2\mu(l_3)^2 \\
    &\qquad \times \prod_{p \mid qk_1k_2k_3l_1l_2l_3} \left( \frac{p}{p+2}\right) \sum_{\substack{p \mid k_4 \Rightarrow p \mid k_1k_2k_3 \\ p \mid l_4 \Rightarrow p \mid l_1l_2l_3}} r(k_1k_2^2k_3^3k_4^3)r(l_1l_2^2l_3^3l_4^3).
\end{split}
\end{align} 
The complete multiplicativity of $r$ implies that
\begin{align*}
 \sum_{p \mid k_4 \Rightarrow p \mid k_1k_2k_3} r(k_1k_2^2k_3^3k_4^3) &= \prod_{p \mid k_1} \sum_{j=0}^\infty r(p^{3j+1}) \prod_{p \mid k_2} \sum_{j=0}^\infty r(p^{3j+2})\prod_{p \mid k_3} \sum_{j=0}^\infty r(p^{3j+3}) \\
 &= \prod_{p \mid k_1} \frac{r(p)}{1-r(p)^3} \prod_{p \mid k_2} \frac{r(p)^2}{1-r(p)^3} \prod_{p \mid k_3} \frac{r(p)^3}{1-r(p)^3},
\end{align*} 
and hence, the right-hand side on \eqref{decom} is 
\begin{align}
\begin{split}
    \label{decom2}
    & \sum_{k_1,l_1 \in S(\xi)} \mu(k_1)^2\mu(l_1)^2 \sum_{\substack{k_2, l_2 \in S(\xi) \\ (k_2,k_1)=1 \\ (l_2,l_1)=1 \\ q^mk_1k_2^2l_1^2l_2=\cube}} \mu(k_2)^2 \mu(l_2)^2 \sum_{\substack{k_3, l_3 \in S(\xi) \\ (k_3,k_1k_2)=1 \\ (l_3,l_1l_2)=1}} \mu(k_3)^2\mu(l_3)^2 \prod_{p \mid qk_1k_2k_3l_1l_2l_3} \left( \frac{p}{p+2}\right)\\
    &\qquad \times \prod_{p \mid k_1} \frac{r(p)}{1-r(p)^3} \prod_{p \mid k_2} \frac{r(p)^2}{1-r(p)^3} \prod_{p \mid k_3} \frac{r(p)^3}{1-r(p)^3}  \prod_{p \mid l_1} \frac{r(p)}{1-r(p)^3} \prod_{p \mid l_2} \frac{r(p)^2}{1-r(p)^3} \prod_{p \mid l_3} \frac{r(p)^3}{1-r(p)^3}.
\end{split}
\end{align}
From the squarefree and coprime conditions of $k_1,k_2,k_3, l_1,l_2,l_3$ ($k_j,l_j$ are all squarefree and $(k_i,k_j)=(l_i,l_j)=1$ for $i \neq j$), $q^mk_1k_2^2l_1^2l_2=\cube$ if and only if
\begin{enumerate}
    \item $m \equiv 1 \pmod 3$;
        \begin{enumerate}
        \item $k_1=l_1$ and $k_2=ql_2$, 
        \item $qk_1=l_1$ and $k_2=l_2$, 
        \item $k_1=ql_1$ and $qk_2=l_2$,
    \end{enumerate}
    \item $m \equiv 2 \pmod 3$;
    \begin{enumerate}
        \item $k_1=ql_1$ and $k_2=l_2$, 
        \item $k_1=l_1$ and $qk_2=l_2$, 
        \item $qk_1=l_1$ and $k_2=ql_2$,
    \end{enumerate}
    \item  $m \equiv 0 \pmod 3$;
    \begin{enumerate}
    \item $k_1=l_1$ and $k_2=l_2$.
\end{enumerate}
\end{enumerate}
In the case of (1-a), we note that $(k_1l_1l_2k_3,q)=1$, $(k_2,q)=q$, and $(l_3,q)=1$ or $q$. Similarly, in the case of (1-b), $(k_1k_2l_2l_3,q)=1$, $(l_1,q)=q$, and $(k_3,q)=1$ or $q$ hold. On the other hand, in the case of (1-c), we have $(k_1,q)=(l_2,q)=q$ and $(l_1k_2k_3l_3,q)=1$. Hence, in the case of (1-a), we find that \eqref{decom2} is 
\begin{align*}
    & \sum_{\substack{k_1\in S(\xi)\\ (k_1,q)=1}} \mu(k_1)^2 \sum_{\substack{l_2 \in S(\xi) \\(ql_2,k_1)=1 \\ (l_2,q)=1}} \mu(l_2)^2 \sum_{\substack{k_3 \in S(\xi) \\ (k_3,qk_1l_2)=1 }} \mu(k_3)^2 \sum_{\substack{l_3 \in S(\xi) \\ (l_3,k_1l_2)=1}} \mu(l_3)^2 \prod_{p \mid qk_1k_3l_2l_3} \left( \frac{p}{p+2}\right)\\
    &\qquad \times \prod_{p \mid k_1} \left(\frac{r(p)}{1-r(p)^3}\right)^2 \prod_{p \mid ql_2} \frac{r(p)^2}{1-r(p)^3} \prod_{p \mid k_3} \frac{r(p)^3}{1-r(p)^3} \prod_{p \mid l_2} \frac{r(p)^2}{1-r(p)^3} \prod_{p \mid l_3} \frac{r(p)^3}{1-r(p)^3} \\
    &=\frac{r(q)^2}{1-r(q)^3}\sum_{\substack{k_1\in S(\xi)\\ (k_1,q)=1}} \mu(k_1)^2 \prod_{p \mid k_1} \left(\frac{r(p)}{1-r(p)^3}\right)^2 \sum_{\substack{l_2 \in S(\xi) \\(ql_2,k_1)=1 \\ (l_2,q)=1}} \mu(l_2)^2 \prod_{p \mid l_2} \left(\frac{r(p)^2}{1-r(p)^3}\right)^2 \\
    &\quad \sum_{\substack{k_3 \in S(\xi) \\ (k_3,qk_1l_2)=1 }} \mu(k_3)^2 \prod_{p \mid k_3} \frac{r(p)^3}{1-r(p)^3} \sum_{\substack{l_3 \in S(\xi) \\ (l_3,k_1l_2)=1}} \mu(l_3)^2 \prod_{p \mid l_3} \frac{r(p)^3}{1-r(p)^3} \prod_{p \mid qk_1k_3l_2l_3} \left( \frac{p}{p+2}\right).
\end{align*}
We divide the above into two parts according to whether $(l_3,q)=1$ or $(l_3,q)=q$. If $(l_3,q)=1$, then \eqref{decom2} is rewritten as
\begin{align*}
    &=\frac{r(q)^2}{1-r(q)^3}\sum_{\substack{k_1\in S(\xi)\\ (k_1,q)=1}} \mu(k_1)^2 \prod_{p \mid k_1} \left(\frac{r(p)}{1-r(p)^3}\right)^2 \sum_{\substack{l_2 \in S(\xi) \\(ql_2,k_1)=1 \\ (l_2,q)=1}} \mu(l_2)^2 \prod_{p \mid l_2} \left(\frac{r(p)^2}{1-r(p)^3}\right)^2 \\
    &\quad \sum_{\substack{k_3 \in S(\xi) \\ (k_3,qk_1l_2)=1 }} \mu(k_3)^2 \prod_{p \mid k_3} \frac{r(p)^3}{1-r(p)^3} \sum_{\substack{l_3 \in S(\xi) \\ (l_3,qk_1l_2)=1}} \mu(l_3)^2 \prod_{p \mid l_3} \frac{r(p)^3}{1-r(p)^3} \prod_{p \mid qk_1k_3l_2l_3} \left( \frac{p}{p+2}\right) \\
    &= \frac{r(q)^2}{1-r(q)^3} \frac{q}{q+2}  \prod_{\substack{p \leq \xi \\ p \neq q}} \left(1+\frac{p}{p+2}\sum_{(j_1,j_2,j_3,j_4) \in \mathcal{J}} Q(p)\right),
\end{align*}
where 
\begin{align*}
    Q(x)=  \left(\frac{r(x)}{1-r(x)^3}\right)^{2j_1} \left(\frac{r(x)^2}{1-r(x)^3}\right)^{2j_2} \left(\frac{r(x)^3}{1-r(x)^3}\right)^{j_3} \left(\frac{r(x)^3}{1-r(x)^3}\right)^{j_4}
\end{align*}
and
\begin{align*}
    \mathcal{J} &= \{ (1,0,0,0), (0,1,0,0), (0,0,1,0), (0,0,0,1), (0,0,1,1)\}.
\end{align*}
The range of $(j_1,j_2,j_3,j_4)$ can easily be deduced, since $(k_1,l_2)=1$ implies $(j_1,j_2) \neq (1,1)$, and $(k_3, k_1l_2)=(l_3,k_1l_2)=1$ implies $j_3,j_4=1$ only if $j_1=j_2=0$. 

On the other hand, if $(l_3,q)=q$, then \eqref{decom2} is 
\begin{align*}
    &=\frac{r(q)^2}{1-r(q)^3}\sum_{\substack{k_1\in S(\xi)\\ (k_1,q)=1}} \mu(k_1)^2 \prod_{p \mid k_1} \left(\frac{r(p}{1-r(p)^3}\right)^2 \sum_{\substack{l_2 \in S(\xi) \\(ql_2,k_1)=1 \\ (l_2,q)=1}} \mu(l_2)^2 \prod_{p \mid l_2} \left(\frac{r(p)^2}{1-r(p)^3}\right)^2 \\
    &\quad \sum_{\substack{k_3 \in S(\xi) \\ (k_3,qk_1l_2)=1 }} \mu(k_3)^2 \prod_{p \mid k_3} \frac{r(p)^3}{1-r(p)^3} \sum_{\substack{l_3 \in S(\xi) \\ (l_3,k_1l_2)=1 \\ (l_3,q)=q}} \mu(l_3)^2 \prod_{p \mid l_3} \frac{r(p)^3}{1-r(p)^3} \prod_{p \mid qk_1k_3l_2l_3} \left( \frac{p}{p+2}\right) \\
    &= \frac{r(q)^2}{1-r(q)^3} \frac{r(q)^3}{1-r(q)^3}\frac{q}{q+2} \prod_{\substack{p \leq \xi \\ p \neq q}} \left(1+\frac{p}{p+2}\sum_{(j_1,j_2,j_3,j_4) \in \mathcal{J}} Q(p)\right).
\end{align*}
Summarizing the two cases of $(l_3,q)=1$ and $(l_3,q)=q$, in the case of (1-a), we obtain \eqref{decom2} is
\begin{align*}
    &\frac{r(q)^2}{1-r(q)^3} \left(1+\frac{r(q)^3}{1-r(q)^3} \right)\frac{q}{q+2} \prod_{\substack{p \leq \xi \\ p \neq q}} \left(1+\frac{p}{p+2}\sum_{(j_1,j_2,j_3,j_4) \in \mathcal{J}} Q(p)\right) \\
    &=\frac{r(q)^2}{(1-r(q)^3)^2} \frac{q}{q+2} \prod_{\substack{p \leq \xi \\ p \neq q}} \left(1+\frac{p}{p+2}\sum_{(j_1,j_2,j_3,j_4) \in \mathcal{J}} Q(p)\right).
\end{align*}

By the same argument as (1-a), in the case of (1-b) we can deduce that \eqref{decom2} is 
\begin{align*}
    & \sum_{\substack{k_1\in S(\xi)\\ (k_1,q)=1}} \mu(k_1)^2 \sum_{\substack{l_2 \in S(\xi) \\(l_2,qk_1)=1}} \mu(l_2)^2 \sum_{\substack{k_3 \in S(\xi) \\ (k_3,k_1l_2)=1 }} \mu(k_3)^2 \sum_{\substack{l_3 \in S(\xi) \\ (l_3,qk_1l_2)=1}} \mu(l_3)^2 \prod_{p \mid qk_1k_3l_2l_3} \left( \frac{p}{p+2}\right)\\
    &\qquad \times \prod_{p \mid k_1} \frac{r(p)}{1-r(p)^3} \prod_{p \mid qk_1} \frac{r(p)}{1-r(p)^3} \prod_{p \mid l_2} \left(\frac{r(p)^2}{1-r(p)^3}\right)^2 \prod_{p \mid k_3} \frac{r(p)^3}{1-r(p)^3}  \prod_{p \mid l_3} \frac{r(p)^3}{1-r(p)^3} \\
    &=\frac{r(q)}{1-r(q)^3}\sum_{\substack{k_1\in S(\xi)\\ (k_1,q)=1}} \mu(k_1)^2 \prod_{p \mid k_1} \left(\frac{r(p)}{1-r(p)^3}\right)^2 \sum_{\substack{l_2 \in S(\xi) \\(l_2,qk_1)=1}} \mu(l_2)^2 \prod_{p \mid l_2} \left(\frac{r(p)^2}{1-r(p)^3}\right)^2 \\
    &\quad \sum_{\substack{k_3 \in S(\xi) \\ (k_3,k_1l_2)=1 }} \mu(k_3)^2 \prod_{p \mid k_3} \frac{r(p)^3}{1-r(p)^3} \sum_{\substack{l_3 \in S(\xi) \\ (l_3,qk_1l_2)=1}} \mu(l_3)^2 \prod_{p \mid l_3} \frac{r(p)^3}{1-r(p)^3} \prod_{p \mid qk_1k_3l_2l_3} \left( \frac{p}{p+2}\right) \\
    &=\frac{r(q)}{(1-r(q)^3)^2} \frac{q}{q+2}  \prod_{\substack{p \leq \xi \\ p \neq q}} \left(1+ \frac{p}{p+2}\sum_{(j_1,j_2,j_3,j_4) \in \mathcal{J}} Q(p)\right).
\end{align*}

Finally, in the case of (1-c), we have
\begin{align*}
    & \sum_{\substack{l_1\in S(\xi)\\ (l_1,q)=1}} \mu(l_1)^2 \sum_{\substack{k_2 \in S(\xi) \\(k_2,ql_1)=1}} \mu(k_2)^2 \sum_{\substack{k_3 \in S(\xi) \\ (k_3,ql_1k_2)=1 }} \mu(k_3)^2 \sum_{\substack{l_3 \in S(\xi) \\ (l_3,ql_1k_2)=1}} \mu(l_3)^2 \prod_{p \mid k_2k_3l_1l_3} \left( \frac{p}{p+2}\right)\\
    &\qquad \times \prod_{p \mid ql_1} \frac{r(p)}{1-r(p)^3} \prod_{p \mid l_1} \frac{r(p)}{1-r(p)^3} \prod_{p \mid k_2} \frac{r(p)^2}{1-r(p)^3} \prod_{p \mid qk_2} \frac{r(p)^2}{1-r(p)^3} \prod_{p \mid k_3} \frac{r(p)^3}{1-r(p)^3}  \prod_{p \mid l_3} \frac{r(p)^3}{1-r(p)^3} \\
    &=\frac{r(q)^3}{(1-r(q)^3)^2} \frac{q}{q+2} \prod_{\substack{p \leq \xi \\ p \neq q}} \left(1+ \frac{p}{p+2}\sum_{(j_1,j_2,j_3,j_4) \in \mathcal{J}} Q(p)\right).
\end{align*}
Combining the three cases (1-a), (1-b) and (1-c), we deduce that \eqref{decom2} is 
\begin{align}
\label{nume}
    \frac{r(q)+r(q)^2+r(q)^3}{(1-r(q)^3)^2} \frac{q}{q+2} \prod_{\substack{p \leq \xi \\ p \neq q}} \left(1+ \frac{p}{p+2}\sum_{(j_1,j_2,j_3,j_4) \in \mathcal{J}} Q(p)\right).
\end{align}
By symmetry, by the same calculation as (1), we can show that \eqref{decom2} is \eqref{nume} in the case of (2). In the last case (3), we obtain
\begin{align}
\begin{split}
\label{nume2}
    & \sum_{k_1\in S(\xi)} \mu(k_1)^2 \sum_{\substack{k_2 \in S(\xi) \\ (k_2,k_1)=1 }} \mu(k_2)^2  \sum_{\substack{k_3, l_3 \in S(\xi) \\ (k_3,k_1k_2)=1 \\ (l_3,k_1k_2)=1}} \mu(k_3)^2\mu(l_3)^2 \prod_{p \mid qk_1k_2k_3l_1l_2l_3} \left( \frac{p}{p+2}\right)\\
    &\qquad \times \prod_{p \mid k_1} \left(\frac{r(p)}{1-r(p)^3}\right)^2 \prod_{p \mid k_2} \left(\frac{r(p)^2}{1-r(p)^3}\right)^2 \prod_{p \mid k_3} \frac{r(p)^3}{1-r(p)^3}   \prod_{p \mid l_3} \frac{r(p)^3}{1-r(p)^3} \\
    &=\frac{q}{q+2} \prod_{p \leq \xi} \left(1+ \frac{p}{p+2}\sum_{(j_1,j_2,j_3,j_4) \in \mathcal{J}} Q(p)\right).
\end{split}
\end{align}

By the same treatment of the numerator in the case of (3), we calculate the denominator of \eqref{ratio}. If $k_1k_2^2l_1^2l_2 = \cube$, then by recalling the squarefree and coprime conditions on $k_j,l_j$ we find that $k_1=l_1$ and $k_2=l_2$. Hence, we obtain
\begin{align}
\begin{split}
\label{deno}
    & \sum_{\substack{k,l=1 \\ kl^2 = \tinycube}}^\infty r(k)r(l) \prod_{p \mid kl^2} \left( \frac{p}{p+2}\right) \\
    &=\sum_{k_1\in S(\xi)} \mu(k_1)^2 \prod_{p \mid k_1} \left(\frac{r(p)}{1-r(p)^3}\right)^2 \sum_{\substack{l_2 \in S(\xi) \\(l_2,k_1)=1}} \mu(l_2)^2 \prod_{p \mid l_2} \left(\frac{r(p)^2}{1-r(p)^3}\right)^2 \\
    &\quad \sum_{\substack{k_3 \in S(\xi) \\ (k_3,k_1l_2)=1 }} \mu(k_3)^2 \prod_{p \mid k_3} \frac{r(p)^3}{1-r(p)^3} \sum_{\substack{l_3 \in S(\xi) \\ (l_3,k_1l_2)=1}} \mu(l_3)^2 \prod_{p \mid l_3} \frac{r(p)^3}{1-r(p)^3} \prod_{p \mid k_1k_3l_2l_3} \left( \frac{p}{p+2}\right) \\
    &= \prod_{p \leq \xi} \left(1+ \frac{p}{p+2}\sum_{(j_1,j_2,j_3,j_4) \in \mathcal{J}} Q(p)\right).
\end{split}
\end{align}

In the cases of (1) and (2), we have $q \leq \xi$ since $k_j.l_j$ are both $\xi$-smooth numbers and $k_j =ql_j, qk_j=l_j$ hold. Therefore, by \eqref{ratio}, \eqref{nume}, \eqref{nume2} and \eqref{deno}, we obtain
\begin{align*}
    \frac{\mathcal{M}_1}{\mathcal{M}_2} &= \sum_{\substack{q^m \leq \xi \\ m \equiv 1 \bmod 3}} \frac{\log q}{q^m}\frac{r(q)+r(q)^2+r(q)^3}{(1-r(q)^3)^2} \frac{q}{q+2} \left(1+ \frac{q}{q+2}\sum_{(j_1,j_2,j_3,j_4) \in \mathcal{J}} Q(q)\right)^{-1} \\
    &\quad + \sum_{\substack{q^m \leq \xi \\ m \equiv 2 \bmod 3}} \frac{\log q}{q^m}\frac{r(q)+r(q)^2+r(q)^3}{(1-r(q)^3)^2} \frac{q}{q+2} \left(1+ \frac{q}{q+2}\sum_{(j_1,j_2,j_3,j_4) \in \mathcal{J}} Q(q)\right)^{-1} \\
    &\quad +
    \sum_{\substack{q^m \leq y \\ m \equiv 0 \bmod 3}} \frac{\log q}{q^m} \frac{q}{q+2}.
\end{align*}

The function $Q$ is calculated as
\begin{align*}
    \sum_{(j_1,j_2,j_3,j_4) \in \mathcal{J}} Q(q) &= \left(\frac{r(q)}{1-r(q)^3}\right)^{2}+\left(\frac{r(q)^2}{1-r(q)^3}\right)^{2}+\frac{2r(q)^3}{1-r(q)^3} + \left(\frac{r(q)^3}{1-r(q)^3}\right)^2 \\
    &= \frac{r(q)^2+2r(q)^3+r(q)^4-r(q)^6}{(1-r(q)^3)^2}.
\end{align*}
Therefore, we obtain
\begin{align}
\begin{split}
    \label{1-r^3}
    &\frac{r(q)+r(q)^2+r(q)^3}{(1-r(q)^3)^2} \frac{q}{q+2} \left(1+ \frac{q}{q+2}\sum_{(j_1,j_2,j_3,j_4) \in \mathcal{J}} Q(q)\right)^{-1} \\
    &= r(q) \frac{1+r(q)+r(q)^2}{\left(1+\frac{2}{q}\right)(1-r(q)^3)^2+(r(q)^2+2r(q)^3+r(q)^4-r(q)^6)} \\
    &= r(q) \frac{1+r(q)+r(q)^2}{1+r(q)^2+r(q)^4+O\left(\frac{1}{\xi}\right)} \\
    &= \frac{r(q)}{1-r(q)+r(q)^2}\left(1+O\left(\frac{1}{\xi}\right)\right)
\end{split}
\end{align}
since $(1-r(q)^3)^2/q \ll 1/\xi$. Here, the last equality is deduced from the fact $1+x^2+x^4=(1+x+x^2)(1-x+x^2)$. Substituting the above to $\mathcal{M}_1/\mathcal{M}_2$ and using $1-x+x^2 \leq 1$ for $\abs{x} \leq 1$, we obtain
\begin{align*}
    \frac{\mathcal{M}_1}{\mathcal{M}_2} &=\sum_{\substack{q^m \leq \xi \\ m \equiv 1 \pmod 3}} \frac{\log q}{q^m} \frac{r(q)}{1-r(q)+r(q)^2} + \sum_{\substack{q^m \leq \xi \\  m \equiv 2 \pmod 3}} \frac{\log q}{q^m} \frac{r(q)}{1-r(q)+r(q)^2} \\
    &\quad + \sum_{\substack{q^m \leq y \\ m \equiv 0 \pmod 3 \\ m \geq 3}} \frac{\log q}{q^{m}} \frac{q}{q+2} + o(1) \\
    &\geq \sum_{q \leq \xi } \frac{r(q)\log q}{q}+ \sum_{k=2}^\infty \sum_{q}\frac{\log q}{q^k}-2\sum_{q} \frac{\log q}{(q^3-1)(q+2)}+ o(1).
\end{align*}
By using Mertens' estimate, we have
\begin{align}
\label{Mertens}
    \sum_{q \leq \xi} \frac{r(q)\log q}{q} &= \log \xi -\gamma -\sum_{k=2}^\infty \sum_{q} \frac{\log q}{q^k} -1 +o(1),
\end{align}
and hence, we obtain
\begin{align*}
    \frac{\mathcal{M}_1}{\mathcal{M}_2} &\geq \log \xi-\gamma-1 - 2 \sum_{q} \frac{\log q}{(q^3-1)(q+2)}+o\left(1\right) \\
    &= \log \log z +\log\log\log z+\log B-\gamma -1 - 2 \sum_{q} \frac{\log q}{(q^3-1)(q+2)}+o\left(1\right)
\end{align*}
as claimed.

\section{Proof of Theorem \ref{LVlogderL}; positive extreme values}
To prove the existence of large values of $\re (L^\prime/L(1,\chi))$, we slightly modify the resonator as
\begin{align*}
    R_\omega (\chi) = \prod_{p \leq \xi} \left(1-\omega \left(1-\frac{p}{\xi} \right) \chi(p) \right)^{-1} = \sum_{n=1}^\infty r^\omega(n) \chi(n),
\end{align*}
where $r^\omega$ is a completely multiplicative function defined by $r^\omega(p)=\omega r(p)$ and $r$ is defined in \eqref{coeff-res}. From the choice of the resonator, the relation $(1-r^\omega(q)^3)^2/q \ll 1/\xi$ also holds. This plays an important role to deduce a similar result as \eqref{1-r^3}. By the definition, we have $\overline{r^\omega(p)}=\omega^2 r(p)$. We consider the ratio $\mathcal{M}_1^\omega/\mathcal{M}_2^\omega$, where
\begin{align*}
    \mathcal{M}_1^\omega := \sum_{\substack{\chi \in \mathcal{F}_3(z) \\ \chi \notin \mathcal{E}}} \re \left(-\sum_{n \leq y} \frac{\Lambda(n)\chi(n)}{n} \right)\abs{R_\omega(\chi)}^2 \quad \text{ and } \quad \mathcal{M}_2^\omega := \sum_{\substack{\chi \in \mathcal{F}_3(z) \\ \chi \notin \mathcal{E}}} \abs{R_\omega(\chi)}^2.
\end{align*}
Then, by the same treatment as the case of small values, we find that $\mathcal{M}_1^\omega/\mathcal{M}_2^\omega$ is
\begin{align*}
    \frac{\displaystyle - \re \sum_{n \leq y} \frac{\Lambda(n)}{n} \sum_{\substack{k,l=1 \\ nkl^2 = \tinycube}}^\infty r^\omega(k)r^\omega(l) \prod_{p \mid nkl^2} \left( \frac{p}{p+2}\right)}{\displaystyle \sum_{\substack{k,l=1 \\ kl^2 = \tinycube}}^\infty r^\omega(k)r^\omega(l) \prod_{p \mid kl^2} \left( \frac{p}{p+2}\right)}\left(1 + o\left( 1\right)\right).
\end{align*}
Since $r^\omega(p)^3=r(p)^3$, we obtain
\begin{align}
\begin{split}
    \label{ratioOmega}
    \frac{\mathcal{M}_1^\omega}{\mathcal{M}_2^\omega} &= - \re \sum_{\substack{q^m \leq \xi \\ m \equiv 1 \bmod 3}} \frac{\log q}{q^m}\frac{\omega r(q)+\omega^2 r(q)^2+r(q)^3}{(1-r(q)^3)^2} \frac{q}{q+2} \left(1+ \frac{q}{q+2}\sum_{(j_1,j_2,j_3,j_4) \in \mathcal{J}} Q_\omega(q)\right)^{-1} \\
    &\quad - \re \sum_{\substack{q^m \leq \xi \\ m \equiv 2 \bmod 3}} \frac{\log q}{q^m}\frac{\omega r(q)+\omega^2 r(q)^2+r(q)^3}{(1-r(q)^3)^2} \frac{q}{q+2} \left(1+ \frac{q}{q+2}\sum_{(j_1,j_2,j_3,j_4) \in \mathcal{J}} Q_\omega(q)\right)^{-1} \\
    &\quad - 
    \sum_{\substack{q^m \leq y \\ m \equiv 0 \bmod 3}} \frac{\log q}{q^m} \frac{q}{q+2},
\end{split}
\end{align}
where 
\begin{align*}
    \sum_{(j_1,j_2,j_3,j_4) \in \mathcal{J}} Q_\omega(q) 
    &= \frac{\omega^2r(q)^2+2r(q)^3+\omega r(q)^4-r(q)^6}{(1-r(q)^3)^2}.
\end{align*}
Hence, by standard calculation, we get
\begin{align*}
    & \frac{\omega r(q)+\omega^2 r(q)^2+r(q)^3}{(1-r(q)^3)^2} \frac{q}{q+2} \left(1+ \frac{q}{q+2}\sum_{(j_1,j_2,j_3,j_4) \in \mathcal{J}} Q(q)\right)^{-1} \\
    &= \frac{\omega r(q)}{1-\omega r(q)+\omega^2 r(q)^2}\left(1+O\left(\frac{1}{\xi}\right)\right)
\end{align*}
and
\begin{align*}
    \re \frac{\omega r(q)}{1-\omega r(q)+\omega^2 r(q)^2} &=-\frac{r(q)}{2} \frac{1+2r(q)+r(q)^2}{1+r(q)+r(q)^3+r(q)^4} \leq -\frac{r(q)}{2}
\end{align*}
as $(1+2x+x^2)/(1+x+x^3+x^4) \geq 1$ for $0 \leq x \leq 1$. Hence, \eqref{ratioOmega} is
\begin{align*}
    &\geq \frac{1}{2} \sum_{q \leq \xi} \frac{r(q) \log q}{q} + \frac{1}{2} \sum_{k=2}^\infty \sum_q \frac{\log q}{q^k} - \frac{1}{2} \sum_{q} \frac{(3q+10)\log q}{(q^3-1)(q+2)} +o(1) \\
    &= \frac{1}{2} \left( \log \xi -\gamma -1 - \sum_{q} \frac{(3q+10)\log q}{(q^3-1)(q+2)} +o(1) \right).
    \end{align*}
Therefore, 
\begin{align*}
    \max_{\chi \in \mathcal{F}_3(z)} \re\frac{L^\prime}{L}(1,\chi) &\geq \max_{\substack{\chi \in \mathcal{F}_3(z) \\ \chi \notin \mathcal{E}}} \re\frac{L^\prime}{L}(1,\chi) \\&\geq \max_{\substack{\chi \in \mathcal{F}_3(z) \\ \chi \notin \mathcal{E}}} \re \left(-\sum_{n \leq y} \frac{\Lambda(n)\chi(n)}{n} \right) + o(1) \\
    &\geq \frac{\mathcal{M}_1^\omega}{\mathcal{M}_2^\omega} + o(1) \\
    &\geq \frac{1}{2}\left(\log\log z +\log\log\log z +\log B-\gamma -1 - \sum_{q} \frac{(3q+10)\log q}{(q^3-1)(q+2)} +o(1) \right).
\end{align*}

\section{The case of quadratic fields}\label{quadraticbound}
In this section, we give a proof of the conditional bound \eqref{RM-qudratic}. Since
\begin{align*}
    \gamma_{\mathbb{Q}(\sqrt{D})} = \gamma + \frac{L^\prime}{L}(1,\chi_D),
\end{align*}
we show that under GRH for quadratic Dirichlet $L$-functions
\begin{align}
\begin{split}
\label{logderQuadratic}
    \max_{\abs{D} \leq x} \pm \frac{L^\prime}{L}(1,\chi_D) &\geq \log \log x +\log\log\log x - \mathfrak{c}_{\text{quad}}^\pm - \varepsilon,
\end{split}
\end{align}
where $\mathfrak{c}_{\text{quad}}^\pm = C_{\text{quad}}^\pm \pm \gamma$. 
On the other hand, as mentioned in Section \ref{intro}, using Yang's approach, Dong and Li proved the omega result only for the case of $-L^\prime/L(1,\chi_D)$ with 
\begin{align*}
    \mathfrak{c}_{\text{quad}}^- = \log 4+\sum_p \frac{\log p}{p^2-1}+1 + \gamma= 3.5334\dots.
\end{align*}
To improve the constant $\mathfrak{c}_{\text{quad}}^-$ given by Dong and Li, we are required a careful treatment of resonators. Compared to cubic character sums, quadratic character sums have been much studied. To apply the resonance method, we can apply the conditional result for the evaluation of quadratic character sums due to Darbar and Maiti~\cite{DM}. 

\begin{lem}
\label{lem:quadcharsumOrthogo}
Assuming GRH, for $n=n_0n_1^2$ with the squarefree component $n_0$ and the square component $n_1$, we have
\[
\sideset{}{^\prime}{\sum}_{\abs{D} \leq x} \chi_D(n) = \frac{x}{\zeta(2)} \prod_{p\mid n} \frac{p}{p+1} \mathds{1}_{n= \square} + O(x^{\frac{1}{2}+\varepsilon}f(n_0)g(n_1)),
\]
where $\mathds{1}_{n= \square}$ indicates the indicator function of the square numbers, $f(n)$ and $g(n)$ are arithmetic functions defined in \eqref{funcf-g}.
\end{lem}

In order to utilize the resonance method, we need an approximate of $L^\prime/L(1,\chi_D)$. Proposition 2.3 of~\cite{Lam15}.

\begin{lem}
\label{lem:lam}
Let $0 < \delta<1/2$ be fixed, and $D$ be a fundamental discriminant with $\abs{D}$ large. Let $y \geq (\log \abs{D})^{\frac{10}{\delta}}$ be a real number. If $L(s, \chi_D)$ is non-zero for $\re(s)>1-\delta$ and $\im(s) \leq y^\delta$, then we have
\begin{align}
\label{approx-logderivL}
    -\frac{L^\prime}{L}(1,\chi_D) = \sum_{n \leq y}\frac{\Lambda(n)\chi_D(n)}{n} +O\left( y^{-\frac{\delta}{4}}\right).
\end{align}
\end{lem}

\begin{proof}
    This is a special case of Proposition 2.3 in~\cite{Lam15}.
\end{proof}

We put $y=(\log x)^{10}$ and $\delta=\frac{1}{100}$. Let $\mathscr{E}$ be the set of fundamental discriminant such that (\ref{approx-logderivL}) is not valid under the choice of $y$ and $\delta$. By the zero density estimates due to Heath-Brown~\cite{HB};
\begin{align*}
    N(1-\delta,T,\chi_D) \ll (xT)^\varepsilon x^{\frac{3\delta}{1+\delta}} T^{\frac{1+2\delta}{1+\delta}},
\end{align*}
we have $\# \mathscr{E} \ll x^{\frac{3}{11}+\varepsilon}$. 

Let $\xi= B\log x \log \log x$, where $B$ is a positive constant same as Section \ref{section:neg}. Now, we introduce the following resonator. 
\begin{align*}
    R^\pm(D) := \prod_{p \leq \xi} \left(1 \mp\left( 1 -\frac{p}{z}\right)\chi_D(p)\right)^{-1} = \sum_{n=1}^\infty r^\pm(n)\chi_D(n),
\end{align*}
where $r^\pm(p)=\pm(1-p/\xi)$ for $p<\xi$ and $r^\pm(p)=0$ for $p \geq \xi$. The choice of the resonator is slightly different from the resonator of~\cite{DM} because we need to find the large of $\pm L^\prime/L(1,\chi_D)$.

As the same way as the cubic case, we define
\begin{align*}
    \mathscr{M}_1^\pm := \sideset{}{^\prime}\sum_{\substack{\abs{D} \leq x \\ D \notin \mathscr{E}}} \sum_{n \leq y} \frac{\Lambda(n)\chi_D(n)}{n} \abs{R^\pm(D)}^2 \quad \text{ and } \quad \mathscr{M}_2^\pm := \sideset{}{^\prime}\sum_{\substack{\abs{D} \leq x \\ D \notin \mathscr{E}}} \abs{R^\pm(D)}^2.
\end{align*}
By the same argument as \eqref{UBRS}, we have $\abs{R^\pm(D)}^2 \ll x^{2b(1+O(1/\log x))}$. Hence, we have
\begin{align*}
    \mathscr{M}_1^\pm &= \sideset{}{^\prime}\sum_{\abs{D} \leq x} \sum_{n \leq y} \frac{\Lambda(n)\chi_D(n)}{n} \abs{R^\pm(D)}^2 +O\left( x^{\frac{3}{11}+2b+\varepsilon} \right), \\
    \mathscr{M}_2^\pm &= \sideset{}{^\prime}\sum_{\abs{D} \leq x} \ \abs{R^\pm(D)}^2 +O\left( x^{\frac{3}{11}+2b+\varepsilon} \right).
\end{align*}
By applying Lemma \ref{lem:quadcharsumOrthogo}, the above is equal to
\begin{align*}
   \mathscr{M}_1^\pm &= \sum_{n \leq y} \frac{\Lambda(n)\chi_D(n)}{n} \sum_{k=1}^\infty \sum_{l=1}^\infty r^\pm(k)r^\pm(l) \sideset{}{^\prime}\sum_{\abs{D} \leq x} \chi_D(nkl) +O\left( x^{\frac{3}{11}+2b+\varepsilon}\right)\\
    &= \frac{x}{\zeta(2)} \sum_{n \leq y} \frac{\Lambda(n)}{n} \sum_{\substack{k,l=1 \\ nkl = \square}}^\infty r^\pm(k)r^\pm(l) \prod_{p \mid nkl} \left( \frac{p}{p+1}\right) + O\left( x^{\frac{1}{2}+\varepsilon}\sum_{n \leq y} \frac{\Lambda(n)}{n} \sum_{\substack{k,l=1 \\ nkl = \square}}^\infty r^\pm(k)r^\pm(l)g(s_1) \right)\\
    &\quad + O\left( x^{\frac{1}{2}+\varepsilon}\sum_{n \leq y} \frac{\Lambda(n)}{n} \sum_{\substack{k,l=1 \\ nkl \neq \square}}^\infty r^\pm(k)r^\pm(l)f(s_0)g(s_1) \right)+O\left( x^{\frac{3}{11}+2b+\varepsilon}\right),
\end{align*}
where $s_0$ (resp. $s_1$) is the squarefree (resp. square) component of $nkl$. From the same argument as the cubic case, we have
\begin{align*}
x^{\frac{1}{2}+\varepsilon} \sum_{n \leq y}\frac{\Lambda(n)}{n} \sum_{\substack{k, l=1 \\ nkl =\square}}^\infty r^\pm(k)r^\pm(l)g(s_1)  &\ll x^{\frac{1}{2}+2B + 4\varepsilon},
\end{align*}
and
\begin{align*}
 z^{\frac{1}{2}+\varepsilon} \sum_{n \leq y}\frac{\Lambda(n)}{n} \sum_{\substack{k, l=1 \\ nkl  \neq \square}}^\infty r^\pm(k)r^\pm(l)f(s_0)g(s_1) 
&\ll x^{\frac{1}{2}+2B + 5\varepsilon}.
\end{align*}

Therefore, we obtain
\begin{align}
\begin{split}
\label{ratioQuadratic}
    \frac{\mathscr{M}_1^\pm}{\mathscr{M}_2^\pm} =\frac{\displaystyle \sum_{n \leq y} \frac{\Lambda(n)}{n} \sum_{\substack{k,l=1 \\ nkl^2 = \square}}^\infty r^\pm(k)r^\pm(l) \prod_{\substack{ p \mid nkl}} \left( \frac{p}{p+1}\right)}{\displaystyle \sum_{\substack{k,l=1 \\ kl^2 = \square}}^\infty r^\pm(k)r^\pm(l) \prod_{\substack{p \mid kl}} \left( \frac{p}{p+1}\right)}\left(1 + o\left( 1\right)\right).
\end{split}
\end{align}
As in the similar argument as the cubic case, following~\cite{DM}, we decompose $k$ as $k=k_1k_2^2k_3^2$, where $k_1, k_2$ are square-free and coprime in pairs. Similarly, we also decompose $l$ as $l=l_1l_2^2l_3^2$, where $l_1,l_2$ are as well. In this decomposition, we get $p \mid k_3 \Rightarrow p \mid k_1k_2$ and $p \mid l_3 \Rightarrow p \mid l_1l_2$. We also write $n=q^m$, where $q$ is a prime, since $n$ is supported on only prime powers. By these decompositions, we find that if $nkl = \square$ (resp. $kl=\square$), then $q^m k_1l_1 = \square$ (resp. $k_1l_1=\square$). Hence, we have
\begin{align*}
& \sum_{\substack{k,l=1 \\ q^mkl = \square}}^\infty r^\pm(k)r^\pm(l) \prod_{\substack{p \mid q^mkl}} \left( \frac{p}{p+1}\right) \\
&=  \sum_{\substack{k_1,l_1 \in S(\xi) \\ q^mk_1l_1=\square}} \mu(k_1)^2\mu(l_1)^2 \sum_{\substack{k_2, l_2 \in S(\xi) \\ (k_2,k_1)=1 \\ (l_2,l_1)=1 }} \mu(k_2)^2 \mu(l_2)^2 \prod_{\substack{p \mid qk_1k_2l_1l_2}} \left( \frac{p}{p+1}\right)  \sum_{\substack{p \mid k_3 \Rightarrow p \mid k_1k_2 \\ p \mid l_3 \Rightarrow p \mid l_1l_2}} r^\pm(k_1k_2^2k_3^2)r^\pm(l_1l_2^2l_3^2).
\end{align*}
From the complete multiplicativity of $r^\pm$, we get
\begin{align*}
    \sum_{p\mid k_3 \Rightarrow p \mid k_1k_2} r^\pm(k_1k_2^2k_3^2) &= \prod_{p \mid k_1} \sum_{j=0}^\infty r^\pm(p^{2j+1}) \prod_{p \mid k_2} \sum_{j=0}^\infty r^\pm(p^{2j+2}) = \prod_{p \mid k_1} \frac{r^\pm(p)}{1-r^\pm(p)^2} \prod_{p \mid k_2} \frac{r^\pm(p)^2}{1-r^\pm(p)^2},
\end{align*} 
and hence, we obtain
\begin{align}
\begin{split}
\label{decomQuadratic}
    &\sum_{\substack{k,l=1 \\ q^mkl = \square}}^\infty r^\pm(k)r^\pm(l) \prod_{\substack{p \mid q^mkl}} \left( \frac{p}{p+1}\right) \\ 
    &= \sum_{\substack{k_1,l_1 \in S(\xi) \\ q^mk_1l_1=\square}} \mu(k_1)^2\mu(l_1)^2 \sum_{\substack{k_2, l_2 \in S(\xi) \\ (k_2,k_1)=1 \\ (l_2,l_1)=1 \\ }} \mu(k_2)^2 \mu(l_2)^2 \prod_{p \mid qk_1k_2l_1l_2} \left( \frac{q}{q+1}\right)\\
    &\qquad \times \prod_{p \mid k_1} \frac{r^\pm(p)}{1-r^\pm(p)^2} \prod_{p \mid k_2} \frac{r^\pm(p)^2}{1-r^\pm(p)^2} \prod_{p \mid l_1} \frac{r^\pm(p)}{1-r^\pm(p)^2} \prod_{p \mid l_2} \frac{r^\pm(p)^2}{1-r^\pm(p)^2}.
\end{split}
\end{align}
In the quadratic case, $q^mk_1l_1=\square$ if and only if
\begin{enumerate}[(i)]
    \item $m$ is even and $l_1=k_1$,
    \item $m$ is odd and $l_1=qk_1$,
    \item $m$ is odd and $ql_1=k_1$.
\end{enumerate}
In the case of (i), we have
\begin{align}
\begin{split}
\label{numeQuadratic1}
    & \sum_{\substack{k_1\in S(\xi)}} \mu(k_1)^2 \sum_{\substack{k_2 \in S(\xi) \\ (k_2,k_1)=1 }} \mu(k_2)^2 \sum_{\substack{l_2 \in S(\xi) \\ (l_2,k_1)=1}} \mu(l_3)^2 \prod_{\substack{p \mid qk_1k_2l_2}} \left( \frac{p}{p+1}\right)\\
    &\qquad \times \prod_{p \mid k_1} \frac{r^\pm(p)^2}{(1-r^\pm(p)^2)^2} \prod_{p \mid k_2} \frac{r^\pm(p)^2}{1-r^\pm(p)^2} \prod_{p \mid l_2} \frac{r^\pm(p)^2}{1-r^\pm(p)^2}  \\
    &=\frac{q}{q+1} \prod_{\substack{p \leq \xi}} \left(1+\frac{p}{p+1}P(p)\right),
\end{split}
\end{align}
where
\begin{align*}
P(x)= \frac{{r^\pm(x)}^2}{(1-{r^\pm(x)}^2)^2}+\frac{2{r^\pm(x)}^2}{1-{r^\pm(x)}^2}+\frac{{r^\pm(x)}^4}{(1-{r^\pm(x)}^2)^2} = \frac{3{r^\pm(x)}^2-{r^\pm(x)}^4}{(1-{r^\pm(x)}^2)^2}.
\end{align*}

In the cases of (ii) and (iii), by symmetry, we can calculate \eqref{decomQuadratic} as 
\begin{align}
\begin{split}
\label{numeQuadratic23}
    & 2\sum_{\substack{k_1\in S(\xi)\\ (k_1,q)=1}} \mu(k_1)^2 \sum_{\substack{k_2 \in S(\xi) \\ (k_2,k_1)=1 }} \mu(k_2)^2 \sum_{\substack{l_2 \in S(\xi) \\ (l_2,qk_1)=1}} \mu(l_3)^2 \prod_{p \mid qk_1k_2l_2} \left( \frac{p}{p+1}\right)\\
    &\qquad \times \prod_{p \mid k_1} \frac{r^\pm(p)}{1-r^\pm(p)^2} \prod_{p \mid qk_1} \frac{r^\pm(p)}{1-r^\pm(p)^2} \prod_{p \mid k_2} \frac{r^\pm(p)^2}{1-r^\pm(p)^2} \prod_{p \mid l_2} \frac{r^\pm(p)^2}{1-r^\pm(p)^2} \ \\
    &=\frac{2r^\pm(q)}{1-r^\pm(q)^2}\sum_{\substack{k_1\in S(\xi)\\ (k_1,q)=1}} \mu(k_1)^2 \prod_{p \mid k_1} \left(\frac{r^\pm(p)}{1-r^\pm(p)^2}\right)^2 \sum_{\substack{k_2 \in S(\xi) \\(k_2,k_1)=1}} \mu(k_2)^2 \prod_{p \mid k_2} \frac{r^\pm(p)^2}{1-r^\pm(p)^2} \\
    &\quad \times \sum_{\substack{l_2 \in S(\xi) \\ (l_2,qk_1)=1 }} \mu(l_2)^2 \prod_{p \mid l_2} \frac{r^\pm(p)^2}{1-r^\pm(p)^2} \prod_{p \mid qk_1k_2l_2} \left( \frac{p}{p+1}\right) \\
    &=\frac{2r^\pm(q)}{(1-r^\pm(q)^2)^2}\frac{q}{q+1} \prod_{\substack{p \leq \xi \\ p \neq q}} \left(1+\frac{p}{p+1}P(p)\right).
\end{split}
\end{align}

By the same treatment of the case (i) of the numerator, we calculate the denominator of \eqref{ratioQuadratic} as
\begin{align}
\label{denoQuadratic}
    \sum_{\substack{k,l=1 \\ kl = \square}}^\infty r^\pm(k)r^\pm(l) \prod_{p \mid kl} \left( \frac{p}{p+1}\right) &= \prod_{p \leq \xi} \left(1+ \frac{p}{p+1}P(p)\right).
\end{align}

Therefore, by \eqref{ratioQuadratic}, \eqref{numeQuadratic1}, \eqref{numeQuadratic23} and \eqref{denoQuadratic}, we obtain from the definition of $r^\pm$ that
\begin{align*}
    \frac{\mathscr{M}_1^\pm}{\mathscr{M}_2^\pm} &= \sum_{\substack{q^m \leq \xi \\ m \text{: odd}}} \frac{\log q}{q^m}\frac{2r^\pm(q)}{(1-r^\pm(q)^2)^2} \frac{q}{q+1} \left(1+ \frac{q}{q+1}P(q)\right)^{-1} + \sum_{\substack{q^m \leq y \\ m \text{: even}}} \frac{\log q}{q^m} \frac{q}{q+1}.
\end{align*}
The above can be further calculated as 
\begin{align*}
    \frac{\mathscr{M}_1^\pm}{\mathscr{M}_2^\pm}
    &= \sum_{\substack{q^m \leq \xi \\ m \text{: odd}}} \frac{\log q}{q^m} \frac{2r^\pm(q)}{\left(1+\frac{1}{q}\right)(1-r^\pm(q)^2)^2+(3r^\pm(q)^2-r^\pm(q)^4)}  + \sum_{\substack{q^m \leq y \\ m \text{: even}}} \frac{\log q}{q^m} \frac{q}{q+1}\\
    &= \sum_{\substack{q^m \leq \xi \\ m \text{: odd}}} \frac{\log q}{q^m} \frac{2r^\pm(q)}{1+r^\pm(q)^2}\left(1+O\left(\frac{1}{\xi}\right)\right)  + \sum_{\substack{q^m \leq y \\ m \text{: even}}} \frac{\log q}{q^m} \left(1-\frac{1}{q+1}\right)\\
    &= \sum_{q \leq \xi} \frac{\log q}{q} \frac{2r^\pm(q)}{1+r^\pm(q)^2} + \sum_{k=2}^\infty \sum_q \frac{\log q}{q^k} - \sum_{k=1}^\infty \sum_q \frac{\log q}{q^{2k}(q+1)}.
\end{align*}
Since $2/(1+x^2) \geq 1$ for $\abs{x} \leq 1$, we have
\begin{align*}
     \sum_{q \leq \xi} \frac{\log q}{q} \frac{2r^+(q)}{1+r^+(q)^2} & \geq \sum_{q \leq \xi} \frac{\log q}{q} \left(1-\frac{q}{\xi}\right), \\
     \sum_{q \leq \xi} \frac{\log q}{q} \frac{2r^-(q)}{1+r^-(q)^2} & \leq - \sum_{q \leq \xi} \frac{\log q}{q} \left(1-\frac{q}{\xi}\right).
\end{align*}
Therefore, by using \eqref{Mertens}, we obtain
\begin{align*}
    \max_{\abs{D}\leq x} -\frac{L^\prime}{L}(1,\chi_D) &\geq \max_{\substack{\abs{D}\leq x \\ D \notin \mathscr{E}}} -\frac{L^\prime}{L}(1,\chi_D) \\
    &\geq \max_{\substack{\abs{D}\leq x \\ D \notin \mathscr{E}}}\frac{\mathscr{M}_1^+}{\mathscr{M}_2^+} + o(1) \\
    &\geq  \log \log x +\log\log\log x -\mathfrak{c}_{\text{quad}}^- - \varepsilon,
\end{align*}
and 
\begin{align*}
    \max_{\abs{D}\leq x} \frac{L^\prime}{L}(1,\chi_D) &\geq \max_{\substack{\abs{D}\leq x \\ D \notin \mathscr{E}}} \frac{L^\prime}{L}(1,\chi_D) \\
    &\geq - \min_{\substack{\abs{D}\leq x \\ D \notin \mathscr{E}}}\frac{\mathscr{M}_1^-}{\mathscr{M}_2^-} + o(1) \\
    &\geq \log \log x +\log\log\log x -   \mathfrak{c}_{\text{quad}}^+ -\varepsilon.
\end{align*}
This completes the proof of \eqref{logderQuadratic}.

\begin{ack} 
The author would like to thank Sh\={o}ta Inoue and Masahiro Mine for their fruitful discussions and valuable comments. 
The author is  supported by JSPS KAKENHI (Grant Number:26K16969).
\end{ack} 

\bibliographystyle{plain}

\end{document}